\documentclass[a4paper,12pt]{article}

\newenvironment{proof}{\noindent {\bf Proof:}}{\hfill $\Box$}
\newtheorem{theorem}{Theorem}
\newtheorem{lemma}{Lemma}

\newtheorem{corollary}{Corollary}

\newtheorem{problem}{Problem}

\usepackage{amsmath}
\usepackage{amssymb}
\usepackage{amsfonts}
\usepackage{graphicx}

\def\R{\mathbb{R}}
\def\P{\mathbf{P}}
\def\G{\mathbf{G}}
\def\K{\mathbf{K}}
\def\U{\mathbf{U}}
\def\V{\mathbf{V}}
\def\N{\mathbb{N}}
\def\B{\mathbf{B}}
\def\s{\mathbb{S}}
\def\A{\mathbf{A}}

\def\lm{\lambda_{{\rm min}}}
\def\vol{\mathrm{vol}}

\title{\bf Inner approximations for polynomial matrix inequalities
and robust stability regions}

\begin{document}

\author{Didier Henrion$^{1,2,3}$, Jean-Bernard Lasserre$^{1,2,4}$}

\footnotetext[1]{CNRS; LAAS; 7 avenue du colonel Roche, F-31077 Toulouse; France. {\tt henrion@laas.fr}}
\footnotetext[2]{Universit\'e de Toulouse; UPS, INSA, INP, ISAE; UT1, UTM, LAAS; F-31077 Toulouse; France}
\footnotetext[3]{Faculty of Electrical Engineering, Czech Technical University in Prague,
Technick\'a 2, CZ-16626 Prague, Czech Republic}
\footnotetext[4]{Institut de Math\'ematiques de Toulouse, Universit\'e de Toulouse; UPS; F-31062 Toulouse, France.}

\maketitle

\begin{abstract}
Following a polynomial approach, many robust fixed-order controller design problems can be formulated as
optimization problems whose set of feasible solutions is modelled by parametrized polynomial matrix inequalities
(PMI). These feasibility sets are typically nonconvex. Given a parametrized PMI set, we provide
a hierarchy of linear matrix inequality (LMI) problems whose optimal solutions generate
inner approximations modelled by a single polynomial superlevel set. Those inner
approximations converge in a well-defined analytic sense to the nonconvex original feasible set,
with asymptotically vanishing conservatism.
One may also impose the hierarchy of inner approximations to be nested or
convex. In the latter case they do not converge any more to the feasible set, but they can be used in
a convex optimization framework at the price of some conservatism.
Finally, we show that the specific geometry of nonconvex polynomial stability regions
can be exploited to improve convergence of the hierarchy of inner approximations.
\end{abstract}

\begin{center}
\small
{\bf Keywords}: polynomial matrix inequality, linear matrix inequality,
robust optimization, robust fixed-order controller design,
moments, positive polynomials.
\end{center}

\section{Introduction}

Linear system stability can be formulated semialgebraically in the space
of coefficients of the characteristic polynomial. The region of stability
is generally {\it nonconvex} in this space, and this is a major
obstacle when solving fixed-order and/or robust controller design problems.
Using the Hermite stability criterion, these problems can be formulated
as parametrized polynomial matrix inequalities (PMIs) where
parameters account for uncertainties and the decision variables
are controller coefficients. Recent results on real algebraic geometry
and generalized problems of moments can be used to build up a hierarchy
of convex linear matrix inequality (LMI) {\it outer} approximations of the region
of stability, with asymptotic convergence to its convex hull,
see e.g. \cite{hl04} for a software implementation and examples,
and see \cite{hl06} for an application to PMI problems arising
from static output feedback design.

If outer approximations of nonconvex semialgebraic sets can be readily
constructed with these LMI relaxations, {\it inner} approximations are much
harder to obtain. However, for controller design purposes, inner approximations
are essential since they correspond to sufficient conditions
and hence guarantees of stability or robust stability.
In the robust systems control literature,
convex inner approximations of the stability region have been proposed
in the form of polytopes \cite{n06}, ellipsoids \cite{hpas03}
or more general LMI regions \cite{hsk03,hkl06} derived from
polynomial positivity conditions. Interval analysis can also
be used in this context, see e.g. \cite{wj94}.

In this paper we provide a numerical scheme for approximating from inside
the feasible set $\P \subset \R^n$ of a parametrized PMI $P(x,u)\succeq0$ (for some matrix polynomial $P$), that is, the set of points $x$ such that 
$P(x,u)\succeq0$ for {\it all} values of the parameter
$u$ in some specified domain $\U\subset\R^p$ (assumed to be a basic compact semialgebraic set\footnote{A basic semialgebraic set is a set defined
by intersecting a finite number of polynomial superlevel sets.}). 
This includes as a special case the approximation of
the stability region (and the robust stability region) of linear systems. 
The particular case where $P(x,u)$ is affine in $x$ covers parametrized LMIs
with many applications in robust control, as surveyed e.g. in \cite{s06}.

Given a compact set $\B \subset \R^n$ containing $\P$,
this numerical scheme consists of building up a sequence of inner approximations $\G_d\subset\P\subset\B$, $d\in\N$,
which fulfils two essential conditions:
\begin{enumerate}
\item The approximation converges in a {\it well-defined analytic} sense ;
\item Each set $\G_d$ is defined
in a {\it simple} manner, as a superlevel set of a single polynomial. In our mind, this feature is essential
for a successful implementation in practical applications.
\end{enumerate}

More precisely, we provide a hierarchy of inner approximations $(\G_d)$
of $\P$, where each $\G_d = \{x\in\B:g_d(x)\geq0\}$ is a basic semi-algebraic set for some
polynomial $g_d$ of degree $d$. The vector of coefficients of the polynomial $g_d$
is an optimal solution of an LMI problem. When $d$ increases, the convergence of $(\G_d)$
to $\P$ is very strong. Indeed, the Lebesgue volume of $\G_d$ converges to the Lebesgue 
volume of $\P$. In fact, on any (a priori fixed) compact set $\B$,
the sequence $(g_d)$ converges for the $L_1$-norm on $\B$ to the function 
$x\mapsto \lm(x)=\min_{u\in \U} \lm(x,u)$ where $\lm(x,u)$ is the minimum eigenvalue
of the matrix-polynomial $P(x,u)$ associated with the PMI.
Consequently, $g_d\to \lm$ in (Lebesgue) measure on $\B$,
 and $g_{d_k}\to\lm$ almost everywhere and almost uniformly on $\B$,
for a subsequence $(g_{d_k})$. In addition, if one defines the piecewise polynomial
$\bar{g}_d:=\max_{k\leq d}g_k$, then $\bar{g}_d\to\lm$ almost everywhere, almost uniformly and in (Lebesgue) measure on $\B$.

In addition, we can easily enforce that the inner approximations $(\G_d)$ are nested and/or
convex. Of course, for the latter convex approximations, convergence to $\P$ 
is lost if $\P$ is not convex.
However, on the other hand, having a convex inner approximation of $\P$ may reveal to be very useful, e.g.,  for
optimization purposes.

On the practical and computational sides, the quality of the approximation of $\P$ depends heavily on the 
chosen set $\B\supset \P$ on which to make the approximation of the function
$\lm$. The smaller $\B$, the better the approximation. In particular,
it is worth emphasizing that when the set
$\P$ to approximate is the stability or robust stability region of a linear system, then its particular
geometry can be exploited to construct a tight bounding set $\B$. Therefore,
a good approximation of $\P$ is obtained significantly faster than with an arbitrary set $\B$ containing $\P$.

Finally, let us insist that the main goal of the paper is to show
that it is possible to provide a tight and explicit inner approximation with no quantifier,
of nonconvex feasible sets described with quantifiers. Then this new feasible set can 
be used for optimization purposes and we are facing two cases:
\begin{itemize}
\item the convex case: if $f$ and $-g$ are convex polynomials, $\B = \{x \in \R^n : \|x\|_{\infty} \leq1\}$
and $\G = \{x \in \B : g(x) \geq 0\}$ then the optimization problem $\min_x f(x) \,\mathrm{s.t.}\, x \in \G$
is polynomially solvable. Indeed, functions $f(x)$, $g(x)$, $\|x\|_{\infty}$ are polynomially computable,
of polynomial growth, and the feasible set is polynomially bounded. Then polynomial solvability of the
problem follows from \cite[Theorem 5.3.1]{bn01}.
\item the nonconvex case: if $-g$ is not convex then notice that firstly we still have an optimization problem
with no quantifier, a nontrivial improvement. Secondly we are now faced with an
polynomial optimization problem with a single polynomial constraint and possibly bound constraints
$x \in \B$. One may then apply the hierarchy of convex LMI relaxations described in
\cite[Chapter 5]{l09}. Of course, in general, polynomial optimization is NP-hard. However, if the size of
the problem is relatively small and the degree of $g$ is small, practice seems to reveal that the problem
is solved exactly with few relaxations in many cases, see \cite[\S 5.3.3]{l09}. In addition, if some
structured sparsity in the data is present then one may even solve problems of potentially large size by using
an appropriate sparse version of these LMI relaxations as described in \cite{wkkm06}, see also \cite[\S 4.6]{l09}.
\end{itemize}

The outline of the paper is as follows. In Section \ref{problem}
we formally state the problem to be solved. In Section \ref{lmi}
we describe our hierarchy of inner approximations. In Section
\ref{control}, we show that the specific geometry of the
stability region can be exploited, as illustrated on several
standard problems of robust control. The final section
collects technical results and the proofs.

\section{Problem statement}\label{problem}

Let $\R[x]$ denote the ring or real polynomials in the variables $x=(x_1,\ldots,x_n)$, and 
let $\R[x]_d$ be the vector space of real polynomials of degree at most $d$.
Similarly, let $\Sigma[x]\subset\R[x]$ denote the convex cone of real polynomials that are sums of squares (SOS) of polynomials,
and $\Sigma[x]_d\subset\Sigma[x]$ its subcone of SOS polynomials of degree at most $2d$.
Denote by $\s^m$ the space of $m\times m$ real symmetric matrices. For a given matrix $A\in\s^m$, the notation
$A \succeq 0$ means that $A$ is positive semidefinite, i.e., all its eigenvalues 
are real and nonnegative.  

Let  $P:\R[x,u]\to\s^m$ be a matrix polynomial, i.e. a matrix whose entries are
scalar multivariate polynomials of the vector indeterminates $x$ and $u$. Then
\begin{equation}\label{setp}
{\mathbf P} := \{x \in {\mathbb R}^n \: :\: \forall u \in {\mathbf U}, \: P(x,u) \succeq 0\}
\end{equation}
defines a parametrized polynomial matrix inequality (PMI) set, where
$x \in {\mathbb R}^n$ is a vector of decision variables, $u \in {\mathbb R}^p$
is a vector of uncertain parameters belonging to a compact
semialgebraic set
\begin{equation}\label{setu}
{\mathbf U} := \{u \in {\mathbb R}^p \: :\: a_i(u) \geq 0, \: i=1,\ldots,n_a\}
\end{equation}
described by given polynomials $a_i(u) \in {\mathbb R}[u]$, and $P(x,u)$ is a given symmetric
polynomial matrix of size $m$. 
As $\U$ is compact, without loss of generality we assume that 
for some $i=i^*$, $a_{i^*}(u) = R^2-u^Tu$, where $R$ is sufficiently large.

We also assume that $\mathbf P$ is bounded and that
we are given a compact set $\mathbf B \supset \mathbf P$
with explicitly known moments $y=(y_\alpha)$, $\alpha\in\N^n$, of the Lebesgue measure on $\B$, i.e.
\begin{equation}\label{momb}
y_{\alpha} := \int_{\mathbf B} x^{\alpha} dx
\end{equation}
where $x^{\alpha} := \prod_{i=1}^n x^{\alpha_i}_i$.
Typical choices for $\mathbf B$ are a box or a ball.
To fix ideas, let
\[\B\,:=\,\{x\in\R^n\::\: b_j(x)\geq 0,\:j=1,\ldots,n_b\}\]
for some polynomials $b_j\in\R[x]$.
Again, with no loss of generality, we may and will assume that 
for some $j=j^*$, $b_{j^*}(x) = R^2-x^Tx$, where $R$ is sufficiently large.
Finally, denote by $\vol\,\A$ the 
Lebesgue volume of any Borel set $\A\subset\B$.

We are now ready to state our polynomial inner approximation problem.

\begin{problem}
{\bf (Inner Approximations)}\label{inner}
Given set $\mathbf P$, build 
up a sequence of basic closed semialgebraic sets ${\mathbf G}_d=\{x\in\B\,:\,g_d(x)\geq0\}$, for some 
$g_d\in\R[x]$, such that
\[{\mathbf G}_d \subseteq \mathbf P, \quad d=1,2,\ldots
\quad\mbox{and}\quad \lim_{d \rightarrow \infty} \vol\,{\mathbf G}_d\,=\,\vol\,\P. 
\]
\end{problem}

In addition, we may want the sequence of inner approximations to satisfy
additional nesting or convexity conditions.

\begin{problem}{\bf (Nested Inner Approximations)}\label{nest}
Solve Problem \ref{inner} with the additional constraint
\[\G_{d}\, \subseteq \,\G_{d+1}\,\subseteq\,\P, \quad d=1,2,\ldots\]
\end{problem}

\begin{problem}{\bf (Convex Inner Approximations)}\label{convex}
Given set $\mathbf P$, build up a sequence of nested basic closed convex semialgebraic sets ${\mathbf G}_d=\{x\in\B\,:\,g_d(x)\geq0\}$,
 for some $g_d\in\R[x]$, such that
 \[\G_{d}\, \subseteq \,\G_{d+1}\,\subseteq\,\P, \quad d=1,2,\ldots\]

\end{problem}

\section{A hierarchy of semialgebraic inner approximations}\label{lmi}

Given a polynomial matrix $P(x,u)$ which defines the set $\mathbf P$ in (\ref{setp}),
polynomials $a_i\in\R[u]$ which define the uncertain set $\mathbf U$ in (\ref{setu}),
let ${\mathbf V} = \{v \in {\mathbb R}^m \: :\: v^Tv = 1\}$ denote the
Euclidean unit sphere of $\R^m$ and let $\lm:\B\to\R$ be the function:
\begin{equation}\label{funp}
x\mapsto \lm(x) = \min_{u \in \U} \min_{v \in \V} v^T P(x,u) v
\end{equation}
as the robust minimum eigenvalue function of $P(x,u)$.
Function $\lm$ is continuous but not
necessarily differentiable.
It allows to define set $\mathbf P$
alternatively as the superlevel set
\[
{\mathbf P} = \{x \in {\mathbb R}^n \: :\: \lm(x) \geq 0\}.
\]

\subsection{Primal SOS SDP problems}

Let $a_0\in\R[u]$ be the constant polynomial $1$.
Let $2d_0\geq \max(2+{\rm deg}\,P,\max_i {\rm deg} a_i, \max_j {\rm deg} b_j)$, and
consider the hierarchy of convex optimization problems
indexed by the parameter $d\in\N$, $d\geq d_0$:
\begin{equation}\label{sdp}
\begin{array}{rl}
\rho_d\,=\,\displaystyle\int_\B\lm(x)\,dx\,-\,\min_{g,r,s,t} & \displaystyle\int_\B g(x)\,dx\\[1em]
\mathrm{s.t.} & v^T P(x,u) v - g(x) \,=\, r(x,u,v) (1-v^T v) \\
& +\displaystyle\sum_{i=0}^{n_a} s_i(x,u,v) a_i(u)+\displaystyle\sum_{j=1}^{n_b} t_j(x,u,v) b_j(x) \quad\forall (x,u,v)\\
\end{array}
\end{equation}
where decision variables are coefficients of
polynomials $g\in\R[x]_{2d}$, $r\in\R[x,u,v]_{2d_r}$ and coefficients of
SOS polynomials $s_i\in\Sigma[x,u,v]_{d_{s_i}}$, $i=0,1,\ldots,n_a$,
and $t_j\in\Sigma[x,u,v]_{d_{t_j}}$, $j=1,\ldots,n_b$.
Note in particular that the degrees of the polynomials
should be such that $d_r \geq d-1$, $d_{s_i} \geq d-\lceil ({\rm deg} \,a_i)/2\rceil$
and $d_{t_j} \geq d-\lceil ({\rm deg} \,b_j)/2\rceil$. Since higher degree terms
may cancel, the degrees can be chosen strictly greater than these lower bounds.
However, in the experiments described later on in the paper,
we systematically chose the lowest possible degrees.

For each $d\in\N$ fixed, the associated optimization problem (\ref{sdp}) is a semidefinite programming (SDP) problem.
Indeed, stating that the two polynomials in both sides of the equation in (\ref{sdp}) are identical
translates into linear equalities between the coefficients of polynomials $g, r, (s_i), (t_j)$ and stating that some of them
are SOS translates into semidefiniteness of appropriate symmetric matrices. For more details, the interested reader is referred to e.g.
\cite[Chapter 2]{l09}.

\subsection{Dual moment SDP problems}

To define the dual to SDP problem (\ref{sdp}) we must introduce some notations.

With a sequence $y=(y_\alpha)$, $\alpha\in\N^n$,
let $L_y:\R[x]\to\R$ be the linear functional
\[f\quad (=\sum_{\alpha}f_{\alpha}\,x^\alpha)\quad\mapsto\quad
L_y(f)\,=\,\sum_{\alpha}f_{\alpha}\,y_{\alpha},\quad f\in\R[x].\]
With $d\in\N$, the moment matrix of order $d$ associated with $y$
is the real symmetric matrix $M_d(y)$ with rows and columns indexed 
in $\N^n_d$, and defined by
\begin{equation}
\label{moment}
M_d(y)(\alpha,\beta)\,:=\,L_y(x^{\alpha+\beta})\,=\,y_{\alpha+\beta},\qquad\forall\alpha,\beta\in\N^n_d.
\end{equation}
A sequence $y=(y_\alpha)$ has a representing measure if there exists a finite Borel measure $\mu$ on $\R^n$, such that
$y_\alpha=\int x^\alpha d\mu$ for every $\alpha\in\N^n$.

With $y$ as above and $h\in\R[x]$,  the localizing matrix
of order $d$ associated with $y$ and $h$ is the real symmetric matrix $M_d(h\,y)$ with rows and columns indexed by $\N^n_d$, and whose entry $(\alpha,\beta)$ is given by
\begin{equation}
\label{local}
M_d(y)(h\,y)(\alpha,\beta)\,:=\,L_y(h(x)\,x^{\alpha+\beta})\,=\,
\sum_{\gamma}h_\gamma \,y_{\alpha+\beta+\gamma},\quad\forall\alpha,\beta\in\N^n_d.
\end{equation}

With these notations, the dual to SDP problem (\ref{sdp}) is given by:
\begin{equation}
\label{sdp*}
\begin{array}{rl}
\rho^*_d=\displaystyle\int_\B\lm(x)dx\,-\,\min_y &L_y(v^TP(x,u)v)\\
\mbox{s.t.}&M_d(y)\succeq0,\:M_{d-1}((1-v^Tv)\,y)=0\\
&M_{d-d_{a_i}}(a_i\,y)\succeq0,\quad i=0,1,\ldots,n_a\\
&M_{d-d_{b_j}}(b_j\,y)\succeq0,\quad j=1,\ldots,n_b\\
&L_y(x^\alpha)\,=\,\int_\B x^\alpha\,dx,\quad\forall\alpha\in\N^n_{2d}
\end{array}\end{equation}
where $y\in\N^{n+p+m}_{2d}$. 

\subsection{Convergence}

Before stating our main results, let us recall some standard notions of functional analysis.
Let $g : \B \to \R$ be a function of $x$, and let $(g_d)$ denote a sequence of
functions of $x$ indexed by $d \in {\mathbb N}$. 
Lebesgue space $L_1(\B)$ is the Banach space of integrable functions on $\B$ equipped with
the norm
\[
\Vert g\Vert_1=\int_\B \vert g\vert dx.
\]
Regarding sequence $(g_d)$, we use the following notions of convergence in $\B$ when $d\to\infty$:
\begin{itemize}
\item $g_d\to g$ in $L_1$ norm means
$\displaystyle\lim_{d\to\infty} \Vert g-g_d\Vert_1 = 0$;
\item $g_d\to g$ in Lebesgue measure means
that for every $\varepsilon>0$,
\[
\lim_{d\to\infty}\mathrm{vol}\{x : |g(x)-g_d(x)|\geq\varepsilon\}=0;
\]
\item $g_d\to g$ almost everywhere
means that $\lim_{d\to\infty}g_d(x)=g(x)$ pointwise
except possibly for $x \in \A \subset \B$ with $\mathrm{vol}\,\A = 0$;
\item $g_d\to g$ almost uniformly
means that for every given $\varepsilon>0$, there is a set $\A \subset \B$ such that 
$\mathrm{vol}\,\A < \varepsilon$ and $g_d\to g$ uniformly on $\B\setminus\A$;
\item finally, with the notation $g_d\uparrow g$ we mean that $g_d\to g$
while satisfying $g_d(x)\leq g_{d+1}(x)$ for all $d$.
\end{itemize}
For more details on these related notions of convergence, see \cite[\S 2.5]{ash}.

\begin{lemma}
\label{lemma1}
For every $d\geq d_0$, SDP problem (\ref{sdp}) has an optimal solution $g_d\in\R[x]_{2d}$ and 
\begin{equation}\label{lem1-1}
\rho_d\,=\,\int_\B(\lm(x)-g_d(x))\,dx\,=\,\Vert \lm-g_d\Vert_1.
\end{equation}
\end{lemma}

A detailed proof of Lemma \ref{lemma1} can be found in \S \ref{proof-lemma1}. 
In particular we prove that there is no duality gap between SOS SDP problem (\ref{sdp})
and moment SDP problem (\ref{sdp*}), i.e. $\rho_d = \rho^*_d$.

For every $d\geq d_0$, let $\bar{g}_d:\B\to\R$ be the piecewise polynomial
\begin{equation}\label{piecewise}
x\mapsto \bar{g}_d(x):=\max_{d_0\leq k\leq d}\,g_k(x).
\end{equation}
We are now in position to prove our main result. 

\begin{theorem}
\label{thmain}
Let $g_d\in\R[x]_{2d}$ be an optimal solution of SDP problem (\ref{sdp})
and consider the associated sequence $(g_d)\subset L_1(\B)$ for $d\geq d_0$.
Then:

{\rm (a)} $g_d\to\lm$ in $L_1$ norm and in Lebesgue measure;

{\rm (b)} $\bar{g}_d\uparrow \lm$ almost everywhere, almost uniformly
and in Lebesgue measure.
\end{theorem}
A detailed proof of Theorem \ref{thmain} can be found in \S \ref{proof-thmain}.
It relies on the Stone-Weierstrass theorem, Putinar's Positivstellensatz,
Lebesgue's dominated convergence theorem and Egorov's theorem.

\subsection{Polynomial and piecewise polynomial inner approximations}

\begin{corollary}
\label{coro1}
For every $d\geq d_0$, let $g_d\in\R[x]_{2d}$ be an optimal solution of SDP problem (\ref{sdp}), let
$\bar{g}_d$ be the piecewise polynomial defined in (\ref{piecewise}), and let
\begin{equation}
\label{coro1-0}
\G_d \,:=\, \{x \in \B \: :\: g_d(x) \geq 0\},\quad
\bar{\G}_d\,:=\, \{x \in \B \: :\: \bar{g}_d(x) \geq 0\}.\end{equation}
Then
\begin{eqnarray}
\label{coro1-1}
\G_d\subset\P\quad\forall\,d\geq d_0&\mbox{and}& \lim_{d\to\infty} \vol(\P\setminus \G_d) = 0.\\
\label{coro1-2}
\bar{\G}_{d_0}\subseteq\cdots\subseteq\bar{\G}_d\subseteq\cdots\subset\P
&\mbox{and}& \lim_{d\to\infty} \vol(\P\setminus \bar{\G}_d) = 0.
\end{eqnarray}
That is, sequence $(\G_d)$
solves Problem \ref{inner} and sequence $(\bar{\G}_d)$
solves Problem \ref{nest} if piecewise polynomials are allowed.
\end{corollary}
A proof can be found in \S \ref{proof-coro1}.

\subsection{Nested polynomial inner approximations}

We now consider Problem \ref{nest} where $g_d$ is constrained to be a polynomial instead
of a piecewise polynomial.
We need to slightly modify SDP problem (\ref{sdp}). Suppose that at step $d-1$
in the hierarchy we have already obtained an optimal solution $g_{d-1}\in\R[x]_{2d-2}$, such that
$g_{d-1}\geq g_{d_0}$ on $\B$, for all $d_0\leq d-1$.
At step $d$ we now solve SDP problem (\ref{sdp}) with the additional
constraint
\begin{equation}
\label{add1}
g(x)-g_{d-1}(x)\,=\,c_0(x)+\displaystyle\sum_{j=1}^{n_b} c_j(x)b_j(x),\quad\forall x
\end{equation}
with unknown SOS polynomials $c_0\in\Sigma[x]_d$ and $c_j\in\Sigma[x]_{d-d_{b_j}}$.

\begin{corollary}
\label{coro2}
Let $g_d\in\R[x]_{2d}$ be an optimal solution of SDP problem (\ref{sdp}) 
with the additional constraint (\ref{add1}) and let 
$\G_d$ be as in (\ref{coro1-0}) for $d\geq d_0$. Then the sequence $(\G_d)$
solves Problem \ref{nest}.
\end{corollary}
For a proof see \S \ref{proof-coro2}.

\subsection{Convex nested polynomial inner approximations}

Finally, for $g\in\R[x]_{2d}$, denote by $\nabla^2 g(x)$ the Hessian matrix of $g$ at $x$, and 
consider SDP problem (\ref{sdp}) with the additional constraint 
\begin{equation}
\label{add2}
v^T \nabla^2 g(x) v = c_0(x,v) + \sum_{j=1}^{n_b} c_j(x,v)b_j(x)+c_{n_b+1}(x,v) (1-v^T v),\end{equation}
for some SOS polynomials $c_0\in\Sigma[x,v]_d$, $c_j\in\Sigma[x,v]_{d-d_{b_j}}$
and $c_{n_b+1}\in\Sigma[x,v]_{d-1}$.

\begin{corollary}
Let $g\in\R[x]_{2d}$ be an optimal solution of SDP problem (\ref{sdp}) with the additional constraint (\ref{add2}) and let
$\G_d$ be as in (\ref{coro1-0}) for $d\geq d_0$. Then the sequence 
$(\G_d)$ solves Problem \ref{convex}.
\end{corollary}
The proof follows along the same lines as the proof of Corollary \ref{coro2}.

\subsection{Example}

Consider the nonconvex planar PMI set
\[
{\mathbf P} = \{x \in {\mathbb R}^2 \: :\:
P(x) = \left[\begin{array}{cc} 1-16x_1x_2 & x_1 \\
x_1 & 1-x_1^2-x_2^2 \end{array}\right] \succeq 0\}
\]
which is Example II-E in \cite{hl06} scaled
to fit within the unit box
\[
{\mathbf B} = \{x \in {\mathbb R}^2 \: :\: \|x\|_{\infty} \leq 1\}
\]
whose moments (\ref{momb}) are readily given by
\[
y_{\alpha} = \frac{4}{(\alpha_1+1)(\alpha_2+1)}.
\]
\begin{figure}[h!]
\centering
\includegraphics[width=0.49\textwidth]{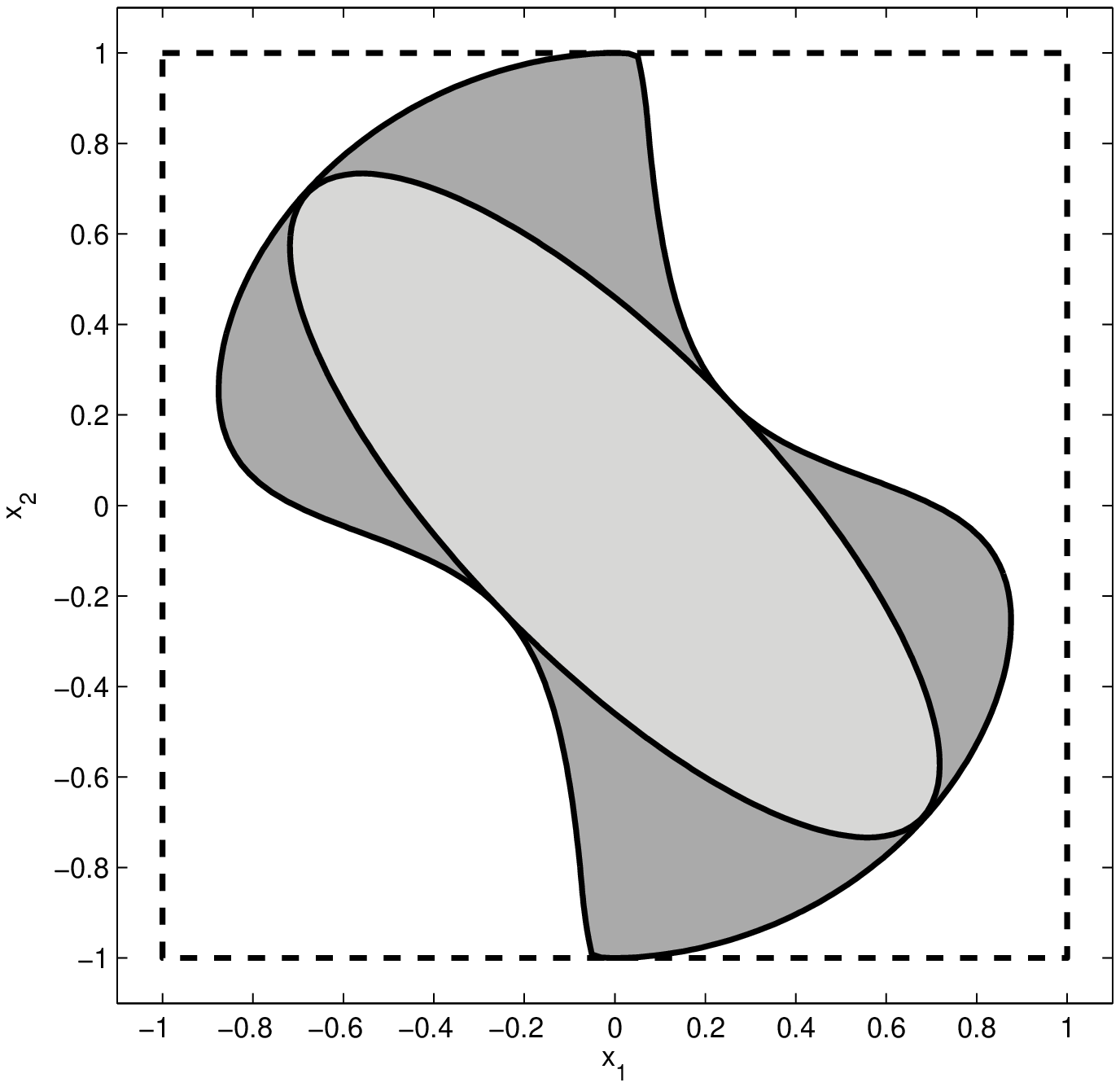}
\includegraphics[width=0.49\textwidth]{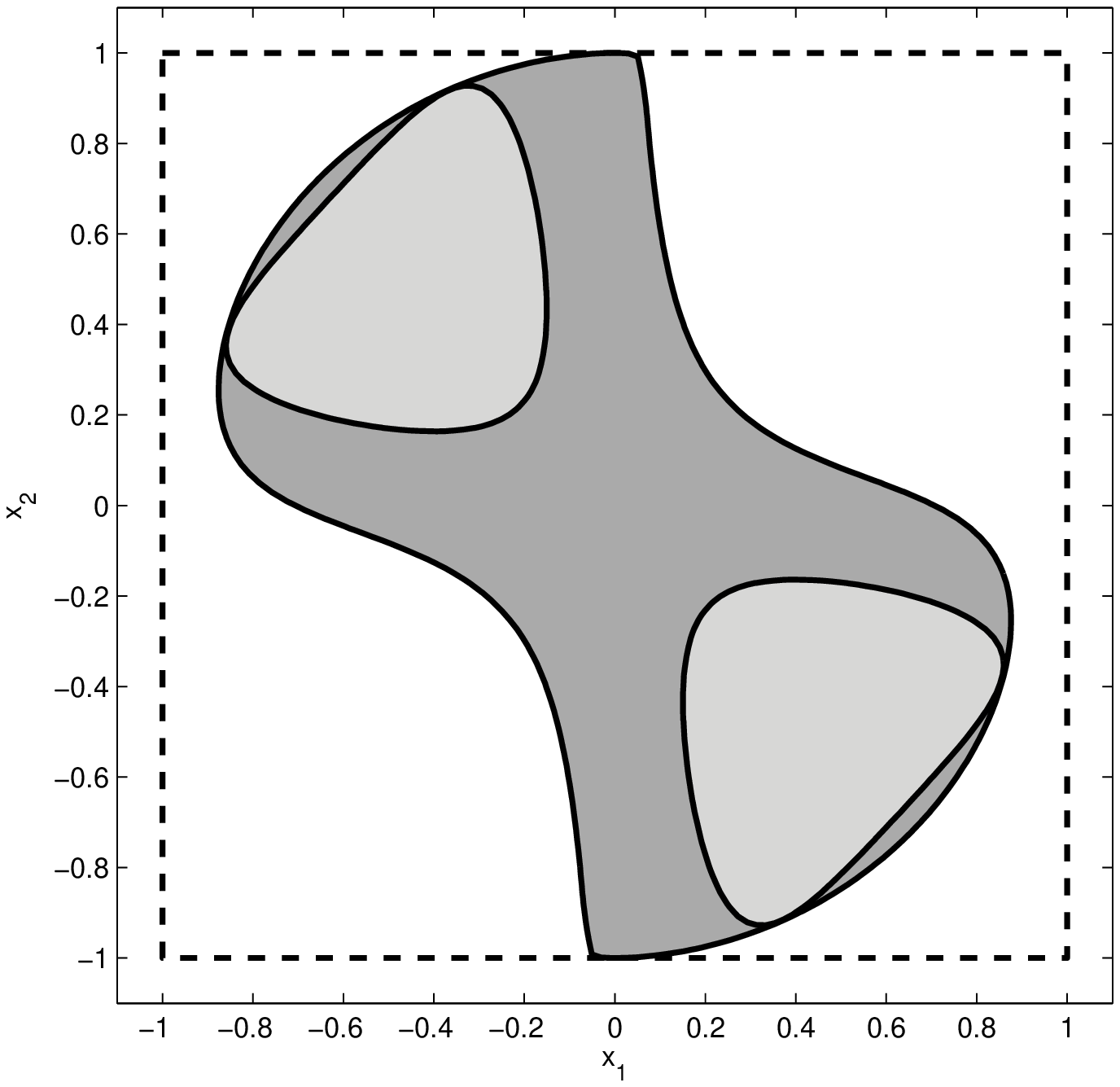}
\caption{Degree two (left) and four (right) inner approximations (light gray)
of PMI set (dark gray) embedded in unit box (dashed).\label{figqmi24}}
\end{figure}

On Figure \ref{figqmi24} we represent the degree two and degree
four solutions to SDP problem (\ref{sdp}), 
modelled by YALMIP 3 and solved by SeDuMi 1.3
under a Matlab environment. We
see in particular that the degree four polynomial superlevel set
${\mathbf G}_2$
is somewhat smaller than expected. This is due to the fact
that the objective function in problem (\ref{sdp})
is the integral of $g(x)$ over the whole box $\mathbf B$,
not only over PMI set $\mathbf P$. There is a significant role
played by the components of the integral on complement
set ${\mathbf B}\backslash{\mathbf P}$, and this deteriorates the
inner approximation. 

This issue can be addressed partly by embedding $\mathbf P$
in a tighter set $\mathbf B$, for example here the unit disk
\[
{\mathbf B} = \{x \in {\mathbb R}^2 \: :\: \|x\|_2 \leq 1\}
\]
whose moments (\ref{momb}) are given by
\[
y_{\alpha} = \frac{\Gamma(\frac{\alpha_1+1}{2})\Gamma(\frac{\alpha_2+1}{2})}
{\Gamma(2+\frac{\alpha_1+\alpha_2}{2})}
\]
where $\Gamma$ is the gamma function such that $\Gamma(k)=(k-1)!$
for integer $k$. See \cite[Theorem 3.1]{lz01} for the general
expression\footnote{Note however that
there is an incorrect factor $2^{-n}$ in the right handside of equation (3.3)
in \cite{lz01}.}
of moments of the unit disk in ${\mathbb R}^n$.
\begin{figure}[h!]
\centering
\includegraphics[width=0.49\textwidth]{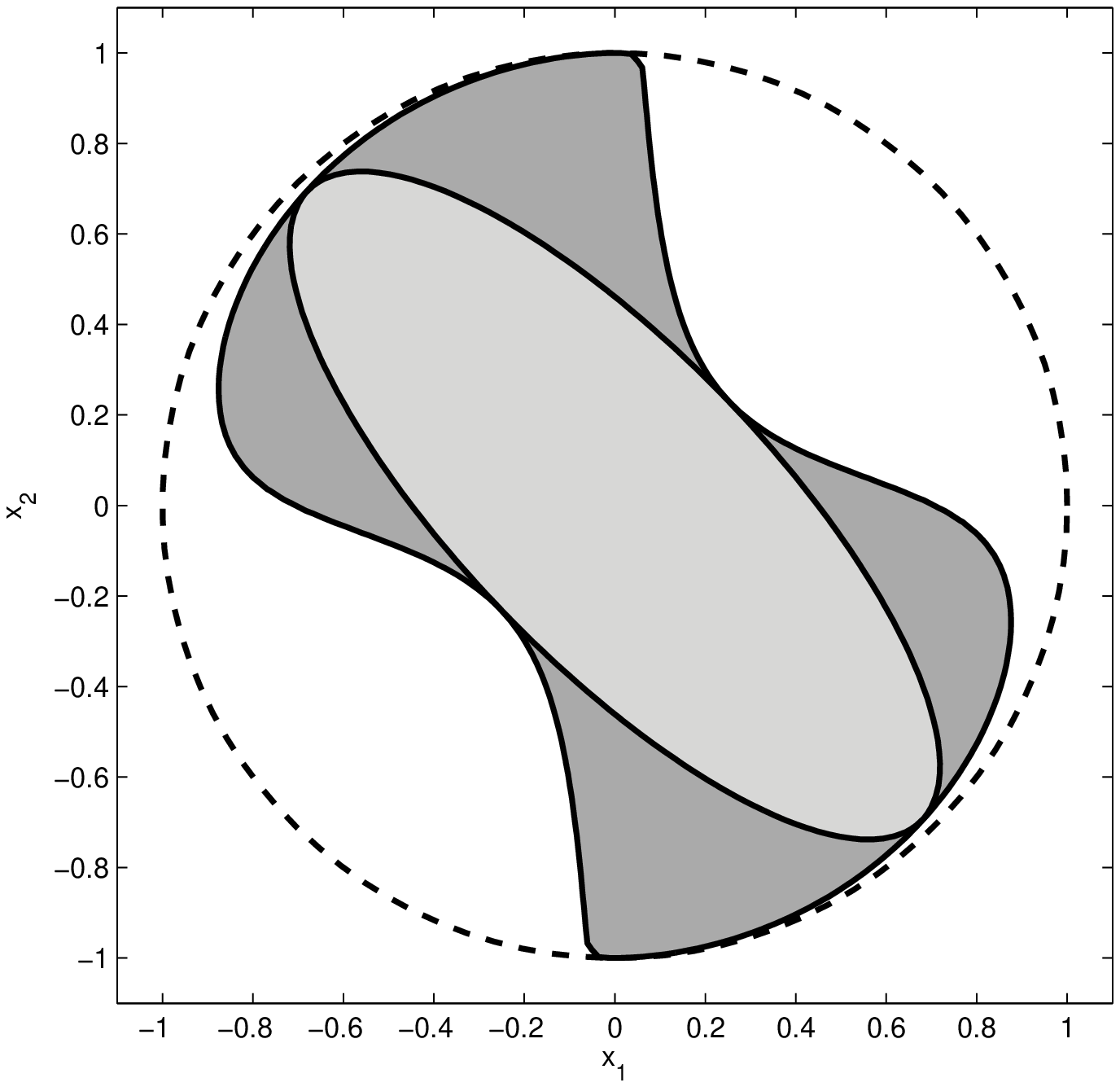}
\includegraphics[width=0.49\textwidth]{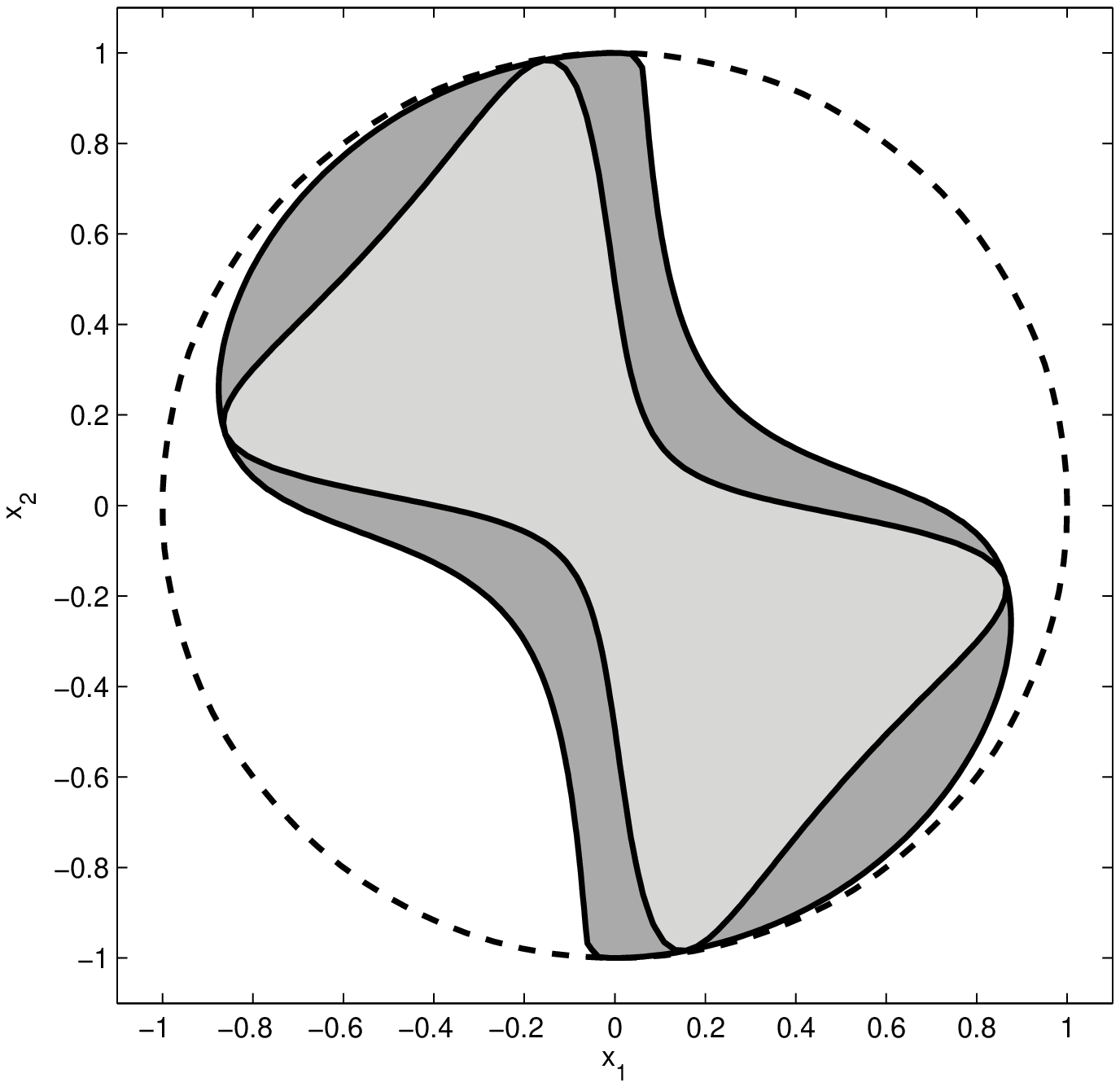}
\caption{Degree two (left) and four (right) inner approximations (light gray)
of PMI set (dark gray) embedded in unit disk (dashed).\label{figqmidisk24}}
\end{figure}
\begin{figure}[h!]
\centering
\includegraphics[width=0.49\textwidth]{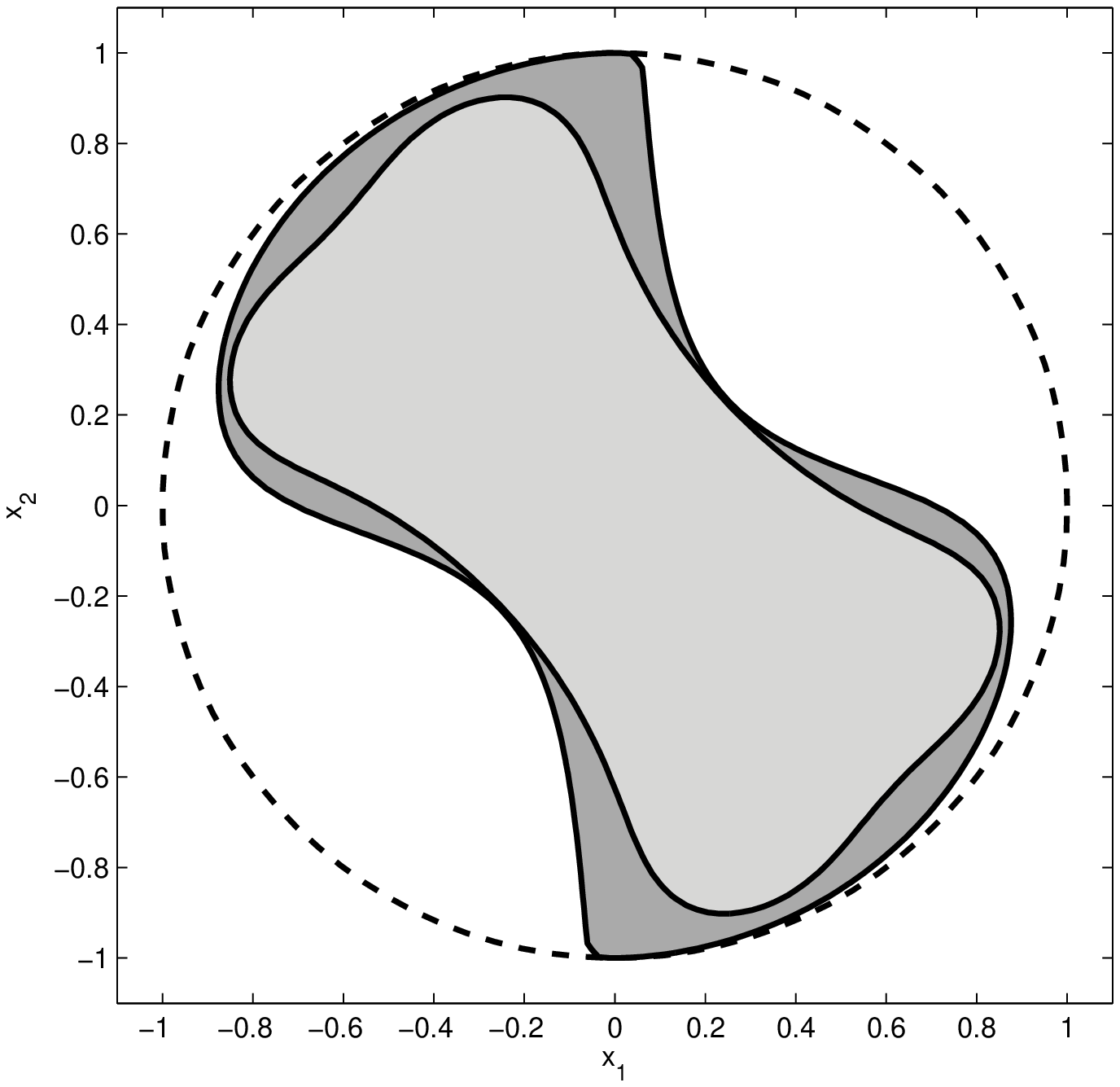}
\includegraphics[width=0.49\textwidth]{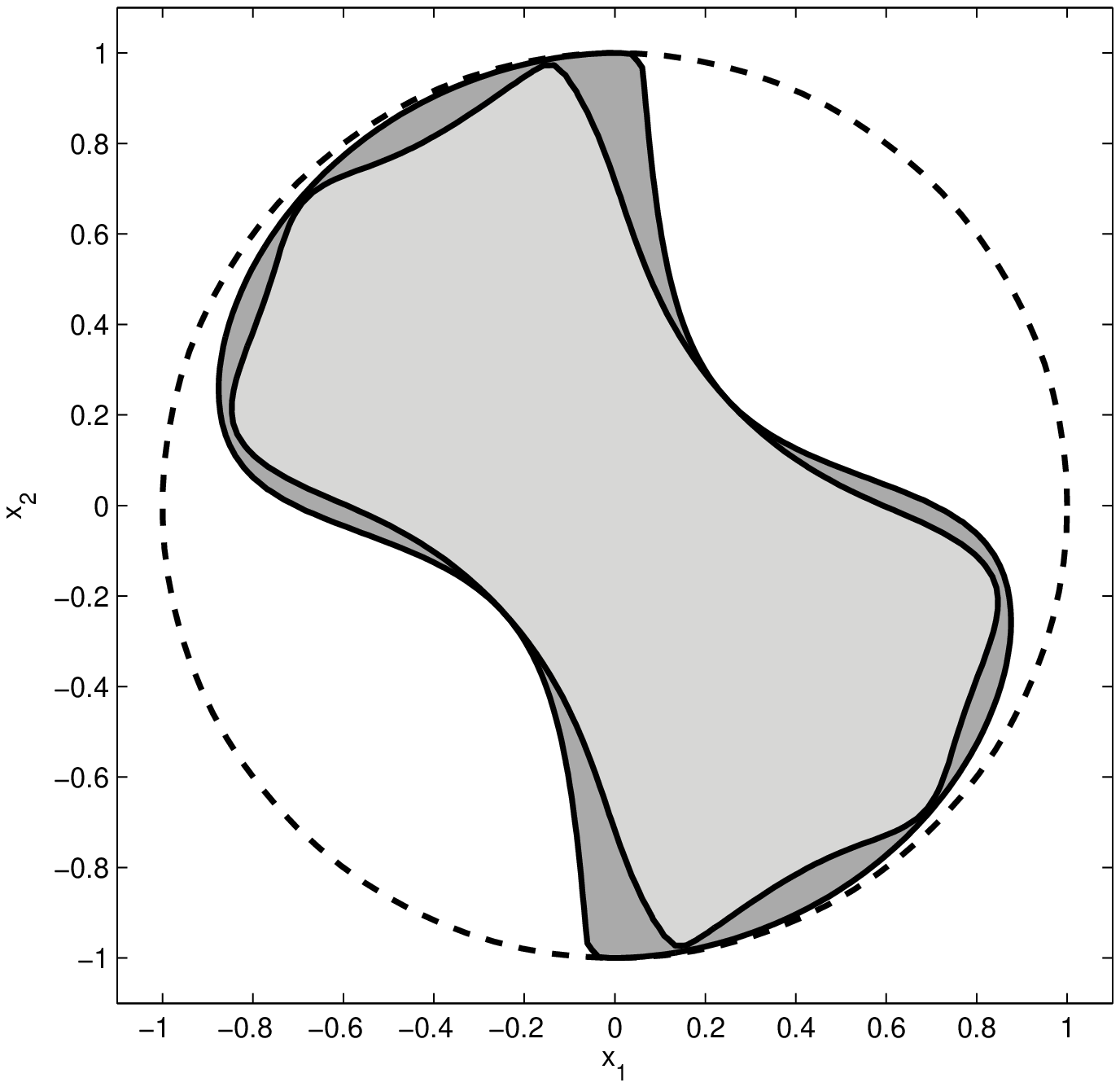}
\caption{Degree six (left) and eight (right) inner approximations (light gray)
of PMI set (dark gray) embedded in unit disk (dashed).\label{figqmidisk68}}
\end{figure}

On Figure \ref{figqmidisk24} we represent the degree two and degree
four solutions to SDP problem (\ref{sdp}). Comparing
with Figure \ref{figqmi24}, we see that
the approximations embedded in the unit disk are much
tighter than the approximations embedded in the unit box.
Finally, on Figure \ref{figqmidisk68} we represent the tighter
degree six and degree eight inner approximations within the unit disk.

\section{Geometry of control problems}\label{control}

As explained in the introduction, inner approximations of the
stability regions are essential for fixed-order controller design.
The PMI regions arising from parametric stability conditions
have a specific geometry that can be exploited to improve the
convergence of the hierarchy of inner approximations.
In this section, we first recall Hermite's PMI formulation
of (discrete-time) stability conditions. Then we recall
that the PMI stability region is the image of a unit box
through a multi-affine mapping, which allows to derive
explicit expressions for the moments of the full-dimensional
stability region, as well as tight polytopic outer approximations
of low-dimensional affine sections of the stability region.
Numerical examples illustrate these techniques for fixed-order
nominal and robustly stabilizing controller design.

\subsection{Hermite's PMI}

Derived by the French mathematician Charles Hermite in 1854,
the Hermite matrix criterion is a symmetric version of
the Routh-Hurwitz criterion for assessing stability of
a polynomial. Originally it was derived for locating the
roots of a polynomial in the open upper half of the complex plane,
but with a fractional transform it can be readily transposed
to the open unit disk and discrete-time stability.
The criterion says that a polynomial $x(z)=z^n+x_1z^{n-1}+\cdots+x_{n-1}z+x_n$
has all its roots in the open unit disk if and only if its Hermite
matrix $P(x)=T^T_1(x)T_1(x)-T^T_2(x)T_2(x)$ is positive definite, where
\[
T_1(x) = \left[\begin{array}{cccc}
1 & x_1 & x_2 \\
0 & 1 & x_1 \\
0 & 0 & 1 \\
& & & \ddots
\end{array}\right] \quad
T_2(x) = \left[\begin{array}{cccc}
x_n & x_{n-1} & x_{n-2} \\
0 & x_n & x_{n-1} \\
0 & 0 & x_n \\
& & & \ddots
\end{array}\right]
\]
are $n$-by-$n$ upper-right triangular Toeplitz matrices, see
e.g. the entrywise formulas of \cite[Theorem 3.13]{barnett}
or the construction explained in \cite{hpas03}.
The Hermite matrix is $n$-by-$n$, symmetric and quadratic in coefficients $x=(x_1,x_2,\ldots,x_n)$,
so that the interior of the PMI set
\[
\P = \{x \in \R^n \: :\: P(x) \succeq 0\}
\]
is the parametric stability domain which is bounded, connected but nonconvex for $n\geq 3$.
Optimal controller design amounts to 
optimizing over semialgebraic set $\P$.

\subsection{Multiaffine mapping of the unit box}\label{multiaffine}

As explained e.g. in \cite{n06} or \cite[\S 3.5]{sd11}
and references therein, stability domain $\P$ can 
also be constructed as the image of the unit box (in the space of so-called reflection
coefficients) through a multiaffine mapping. More explicitly $\P = f(\K)$
where $\K = \{k \in \R^n \: :\: \|k\|_{\infty} \leq 1\}$ and
multiaffine mapping $f : \R^n \rightarrow \R^n$ is defined by
\[
\begin{array}{rcl}
f(k) & = &
\left[\begin{array}{ccccc}0&1&0&0&0\\0&0&1&0&0\\0&0&0&1&0\end{array}\right]
\left[\begin{array}{cccc}1&0&0&k_3\\0&1&k_3&0\\0&k_3&1&0\\k_3&0&0&1\\0&0&0&0\end{array}\right]
\left[\begin{array}{ccc}1&0&k_2\\0&1+k_2&0\\k_2&0&1\\0&0&0\end{array}\right]
\left[\begin{array}{cc}1&k_1\\k_1&1\\0&0\\\end{array}\right]
\left[\begin{array}{c}1\\0\end{array}\right] \\
& = &
\left[\begin{array}{c}k_2k_3+k_1(1+k_2)\\k_2+k_1k_3(1+k_2)\\k_3\end{array}\right]
\end{array}
\]
in the case $n=3$. The general expression of $f$ for other values of $n$ is not
given here for space reasons, but it follows
readily from the construction outlined above.

Using this mapping we can obtain moments (\ref{momb}) of $\B=\P$
analytically:
\begin{equation}\label{mom3}
y_{\alpha} = \int_{\mathbf P} x^{\alpha} dx = \int_{\mathbf K} 
(k_2k_3+k_1(1+k_2))^{\alpha_1} 
(k_2+k_1k_3(1+k_2))^{\alpha_2} 
k_3^{\alpha_3}
\det \nabla f(k) dk
\end{equation}
where $\det \nabla f(k) = (1+k_2)(1-k_3^2)$ is
the determinant of the Jacobian of $f$, in the case $n=3$.
For space reasons we do not give
here the explicit value of $y_{\alpha}$ as a function of $\alpha$, but
it can be obtained by integration by parts. 

Finally, let us mention a
well-known geometric property of $\P$: its convex hull
is a polytope whose vertices correspond to the $n+1$ polynomials
with roots equal to $-1$ or $+1$. For example, when $n=3$, we have
\begin{equation}\label{conv3}
\mathrm{conv}\,\P = \mathrm{conv}\{(-3,3,-1),\,(-1,-1,1),\,(1,-1,-1),\,(3,3,1)\}.
\end{equation}

\subsection{Third degree stability region}

\begin{figure}[h!]
\centering
\includegraphics[width=0.49\textwidth]{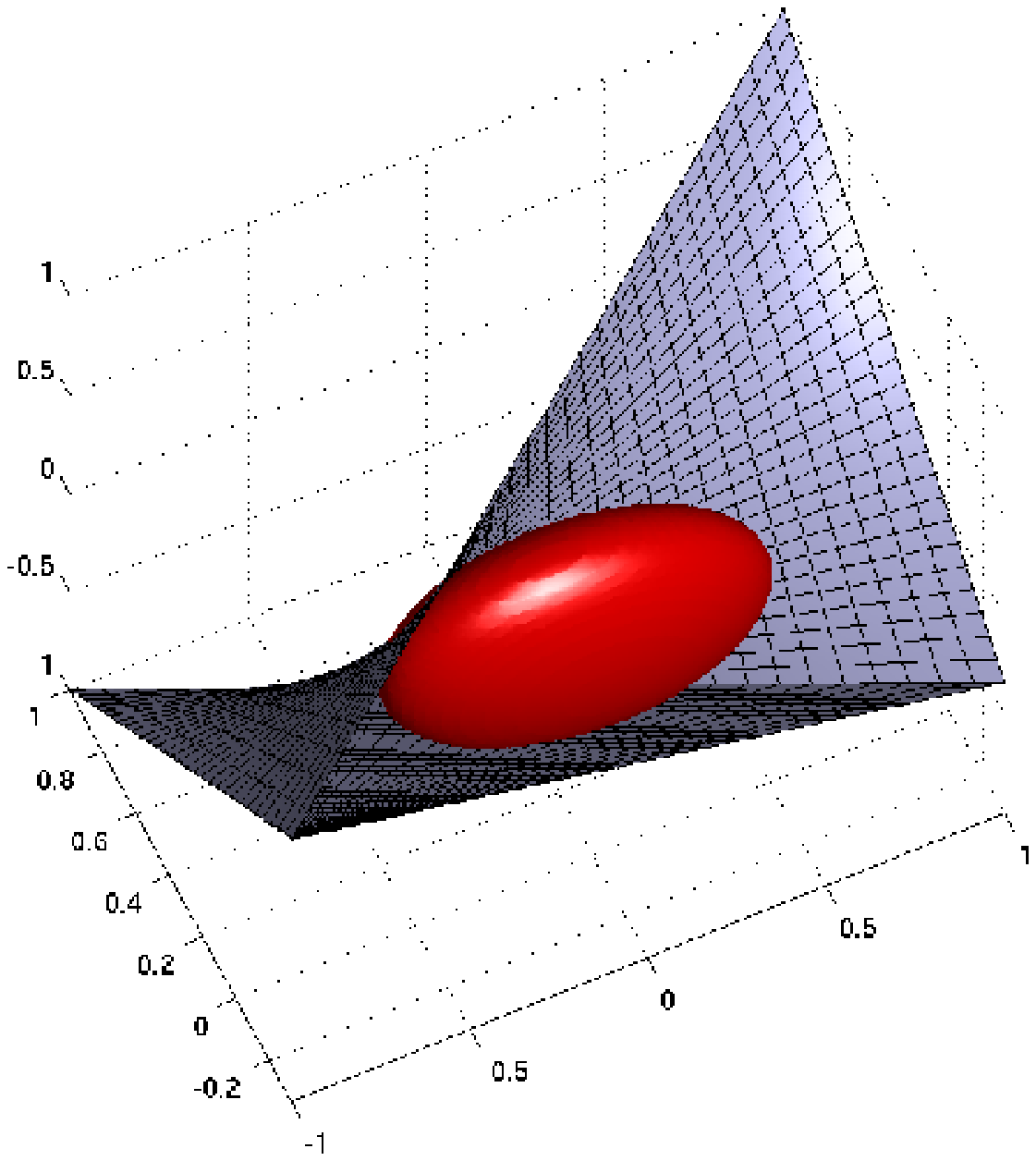}
\includegraphics[width=0.49\textwidth]{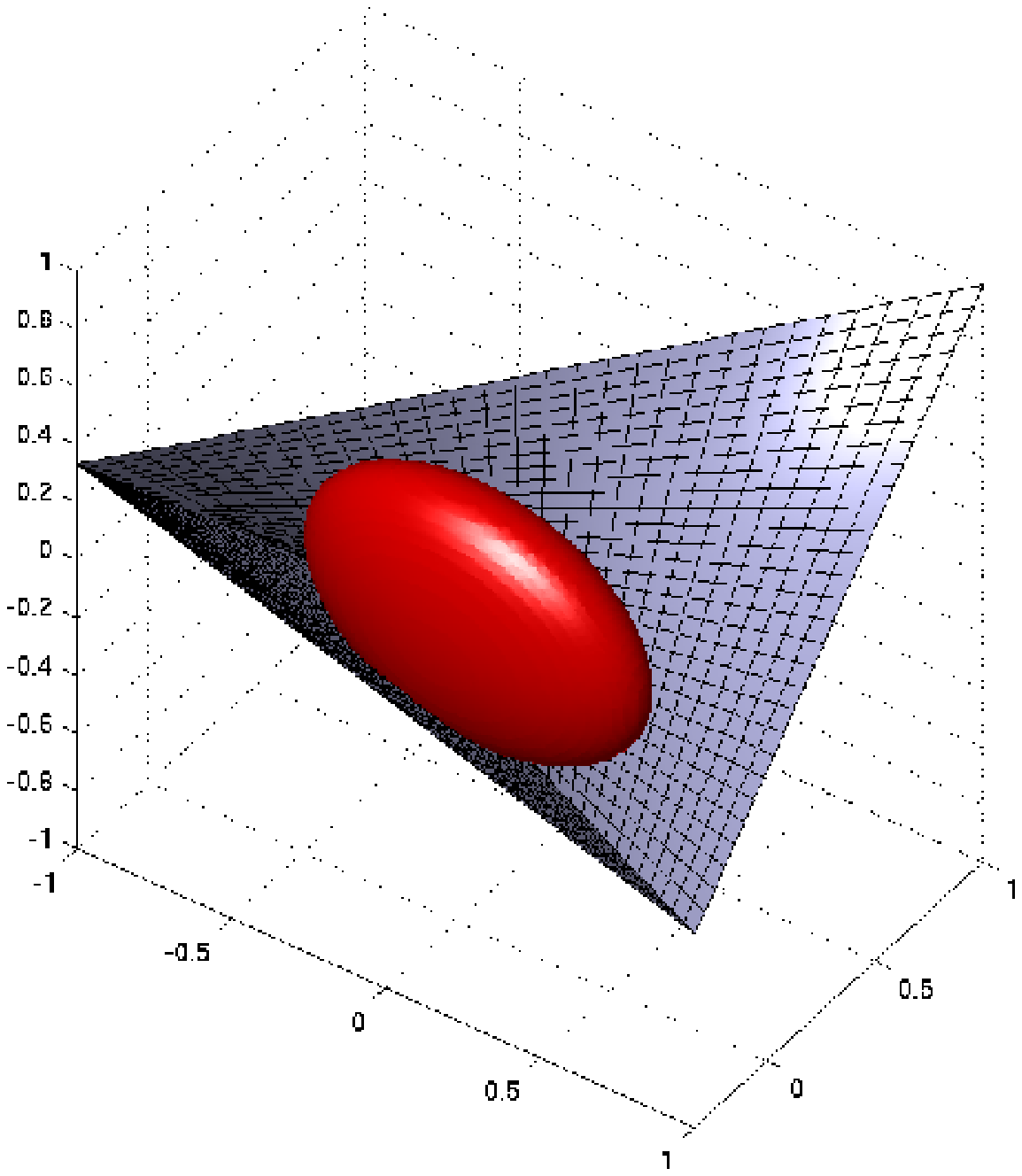}
\caption{Two views of a degree two inner approximation (red) of nonconvex third-degree
stability region (gray).\label{fighermite3d2}}
\end{figure}
\begin{figure}[h!]
\centering
\includegraphics[width=0.49\textwidth]{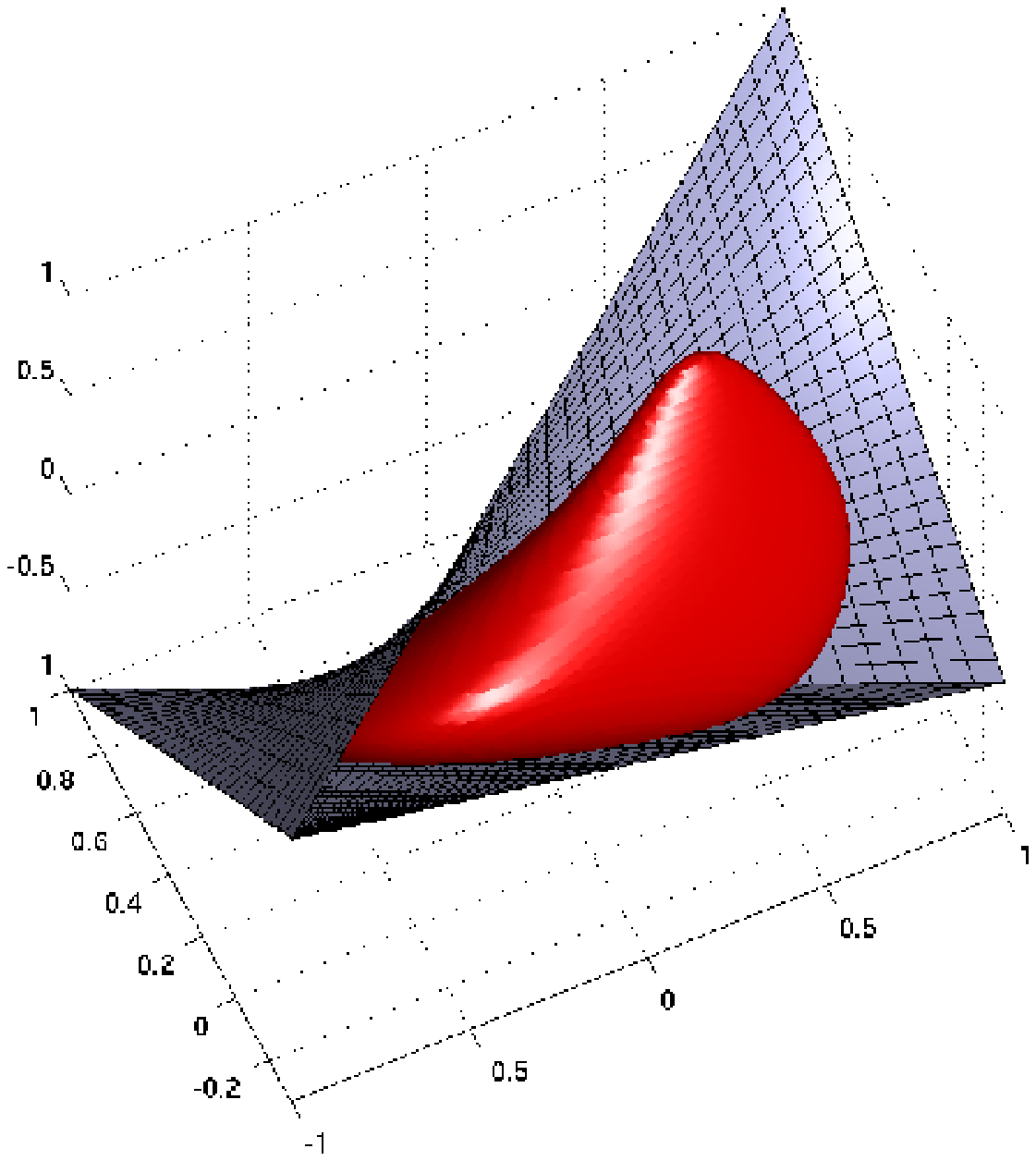}
\includegraphics[width=0.49\textwidth]{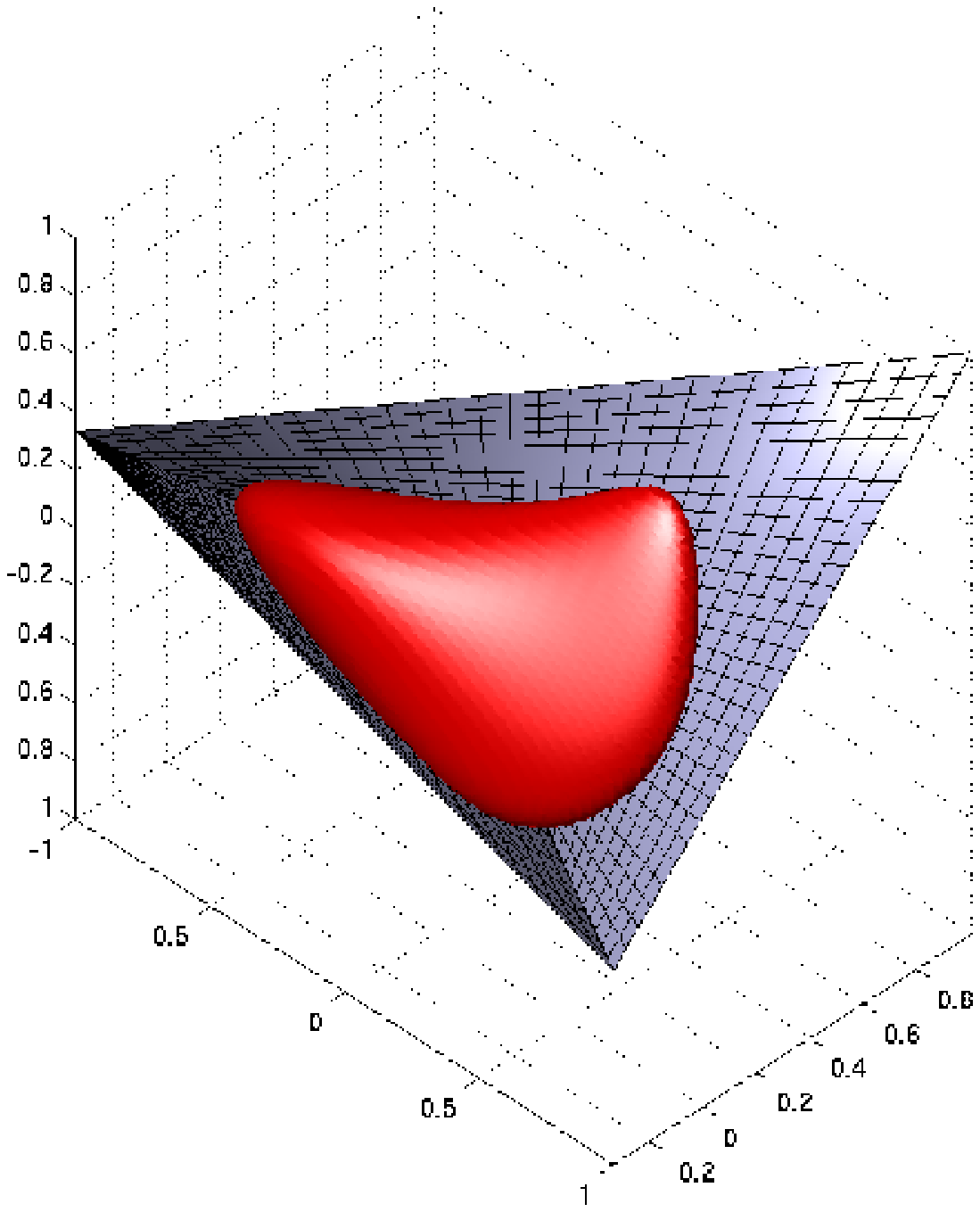}
\caption{Two views of a degree four inner approximation (red) of nonconvex third-degree
stability region (gray).\label{fighermite3d4}}
\end{figure}
\begin{figure}[h!]
\centering
\includegraphics[width=0.49\textwidth]{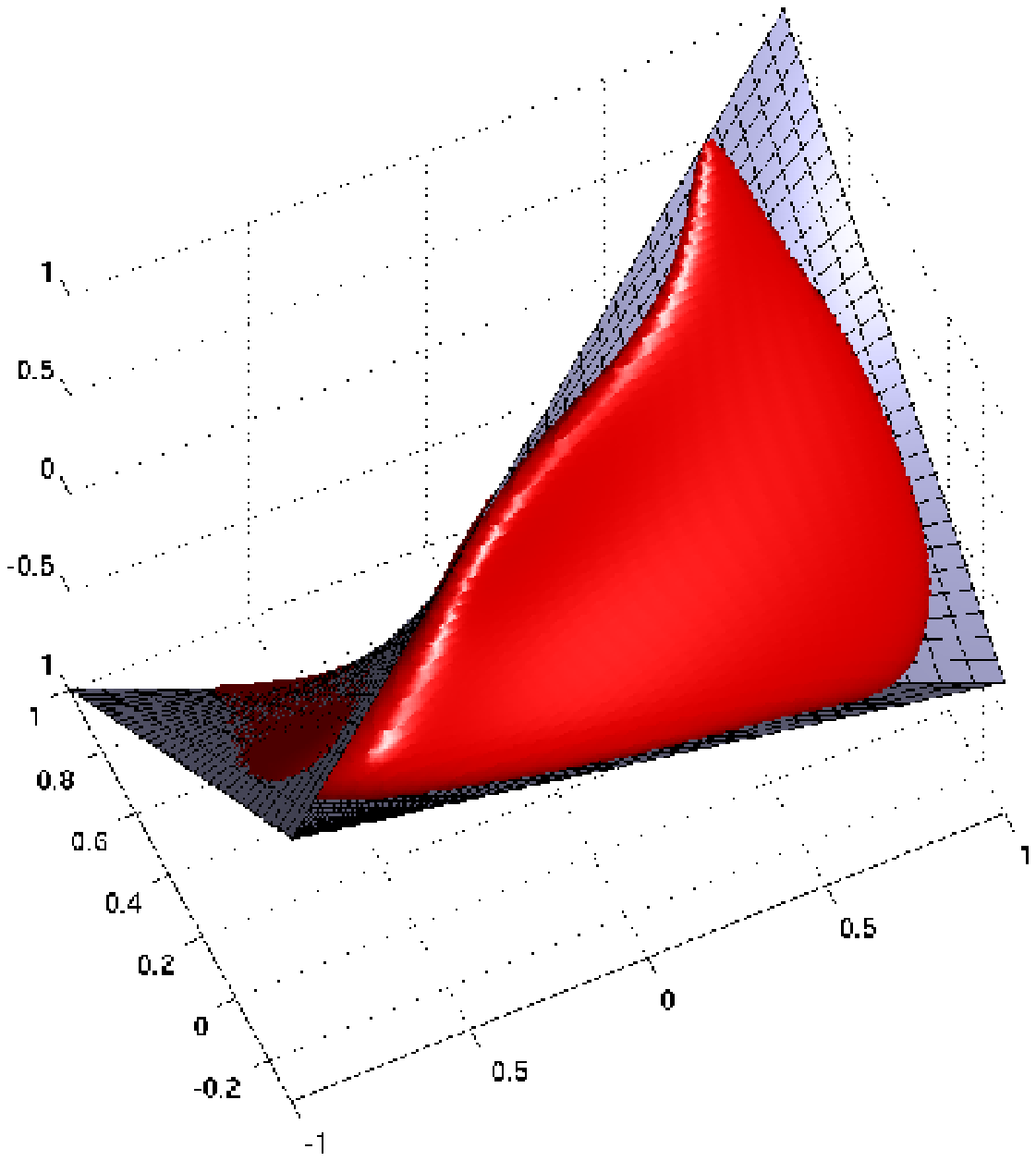}
\includegraphics[width=0.49\textwidth]{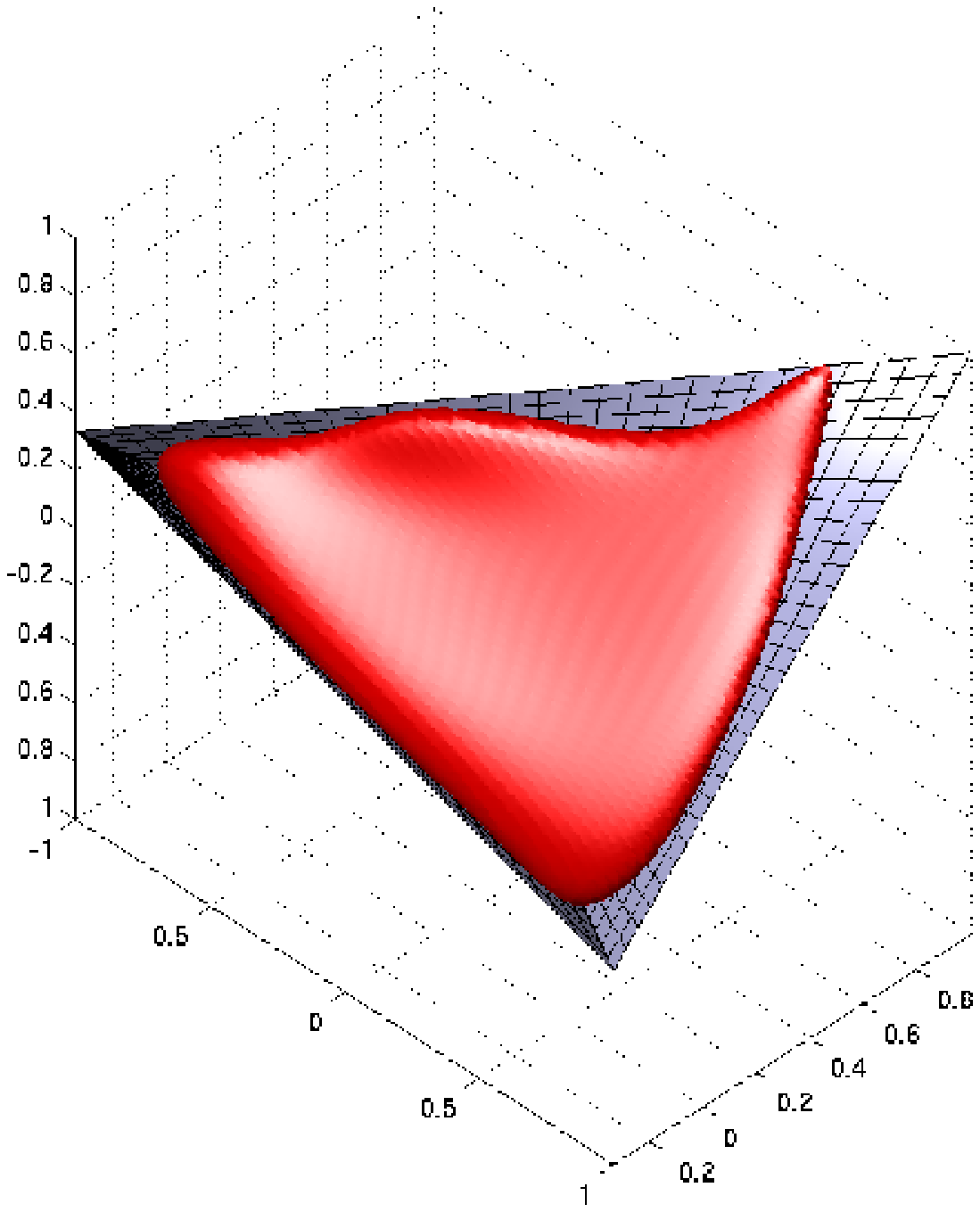}
\caption{Two views of a degree six inner approximation (red) of nonconvex third-degree
stability region (gray).\label{fighermite3d6}}
\end{figure}

Consider the problem of approximating from the inside the nonconvex
stability region $\P$ of a discrete-time third degree polynomial
$z \mapsto z^3+x_1z^2+x_2z+x_3$.
An ellipsoidal inner approximation was proposed in \cite{hpas03}.
The Hermite polynomial matrix defining $\P$ as in (\ref{setp}) is given by
\[
P(x) = \left[\begin{array}{ccc}
1 - x_3^2 & x_1 - x_2x_3 & x_2 - x_1x_3 \\
x_1 - x_2x_3 & 1+x_1^2-x_2^2-x_3^2 & x_1-x_2x_3 \\
x_2 - x_1x_3 & x_1-x_2x_3 & 1-x_3^2 
\end{array}\right].
\]
The boundary of $\P$ consists of two
triangles and a hyperbolic paraboloid. The convex hull of $\P$
is the simplex described in (\ref{conv3}). We have analytic expressions (\ref{mom3})
for the moments (\ref{momb}) of $\B=\P$.

On Figures \ref{fighermite3d2}, \ref{fighermite3d4} and \ref{fighermite3d6}
we respectively represent the degree two, four and six inner approximations
of $\P$, scaled within the unit box for visualization purposes.
We observe that the degree six approximation is very tight,
thanks to the availability of the moments of the Lebesgue measure on $\P$.

\subsection{Fixed-order controller design}\label{design}

Consider the linear discrete-time system with characteristic polynomial
$z \mapsto z^4-(2x_1+x_2)z^3+2x_1z+x_2$ depending affinely on two real
design parameters $x_1$ and $x_2$. It follows from Hermite's stability
criterion
that this polynomial has its roots in the open unit disk if and only if
\[\small
P(x) = \left[\begin{array}{@{\;}c@{\;}c@{\;}c@{\;}c@{\;}}
1-x_2^2 & -2x_1-x_2-2x_1x_2 & 0 & 2x_1+2x_1x_2+x_2^2 \\
-2x_1-x_2-2x_1x_2 & 1+4x_1x_2 & -2x_1-x_2-2x_1x_2 & 0 \\
0 & -2x_1-x_2-2x_1x_2 & 1+4x_1x_2 & -2x_1-x_2-2x_1x_2 \\
2x_1+2x_1x_2+x_2^2 & 0 & -2x_1-x_2-2x_1x_2 & 1-x_2^2
\end{array}\right]
\]
is positive definite. As recalled in (\ref{multiaffine}),
the convex hull of the four-dimensional stability domain
of a degree four polynomial is the simplex with vertices
$(-4,6,-4,1)$, $(-2,0,2,-1)$, $(0,-2,0,1)$, $(2,0,-2,-1)$,
$(4,6,4,1)$ corresponding to the five polynomials with
zeros equal to $-1$ or $+1$. Using elementary linear algebra,
we find out that the image of this simplex
through the affine mapping $(-(2x_1+x_2),0,2x_1,x_2)$
parametrized by $x \in \R^2$ is the two-dimensional simplex
\[
\B = \mathrm{conv}\{(-\frac{1}{4},1),\,(\frac{7}{8},-\frac{1}{2}),\,
(-\frac{5}{8},-\frac{1}{2})\}.
\]
The (closure of the) stability region $\P = \{x \in \R^2 \: :\: P(x) \succeq 0\}$
is therefore included in $\B$, whose moments (\ref{momb}) are
readily obtained e.g. by the explicit formulas of \cite{la01}.

\begin{figure}[h!]
\centering
\includegraphics[width=0.49\textwidth]{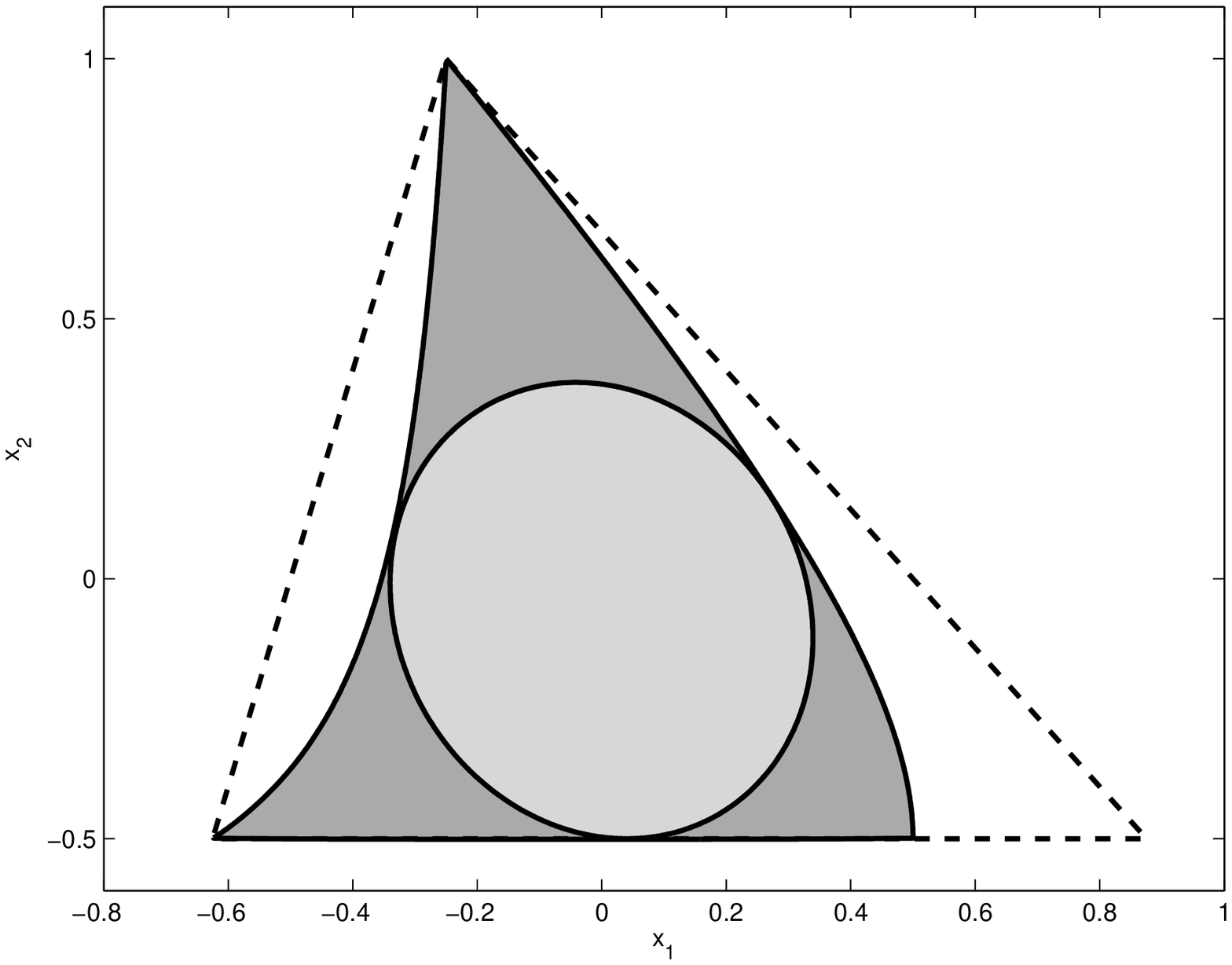}
\includegraphics[width=0.49\textwidth]{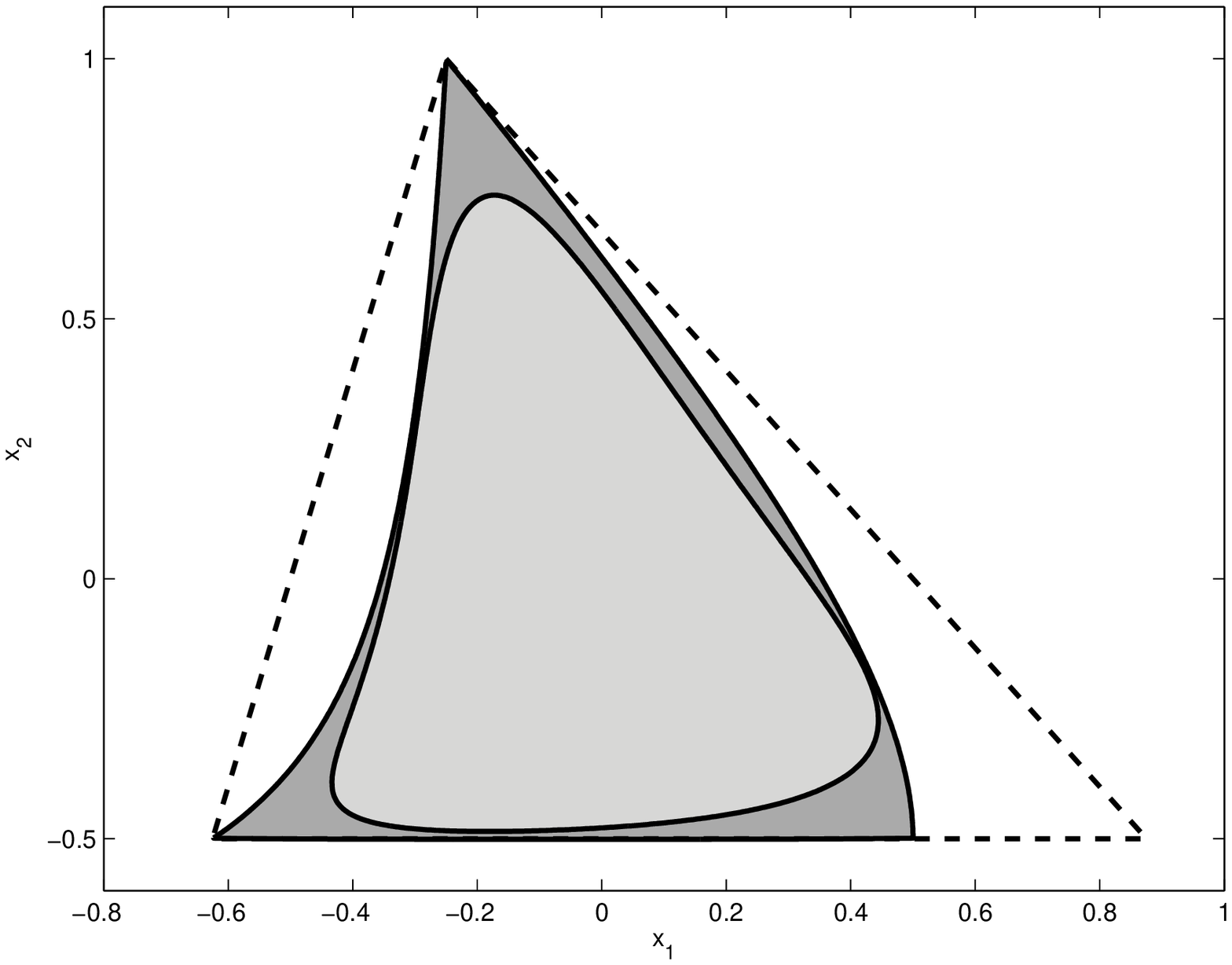}
\caption{Degree two (left) and four (right) inner approximations (light gray)
of PMI stability region (dark gray) embedded in simplex (dashed).\label{fighermite4a}}
\end{figure}
\begin{figure}[h!]
\centering
\includegraphics[width=0.49\textwidth]{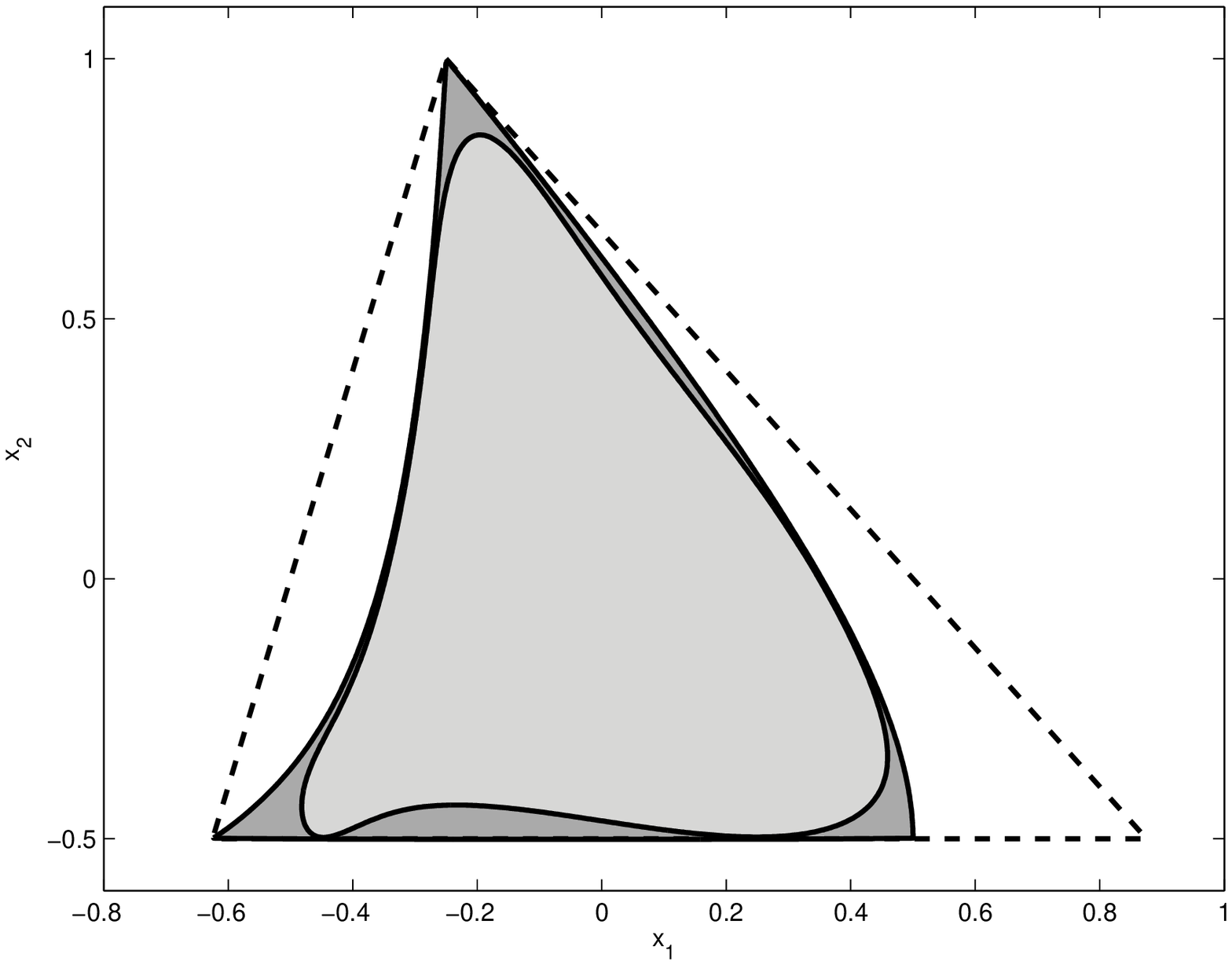}
\includegraphics[width=0.49\textwidth]{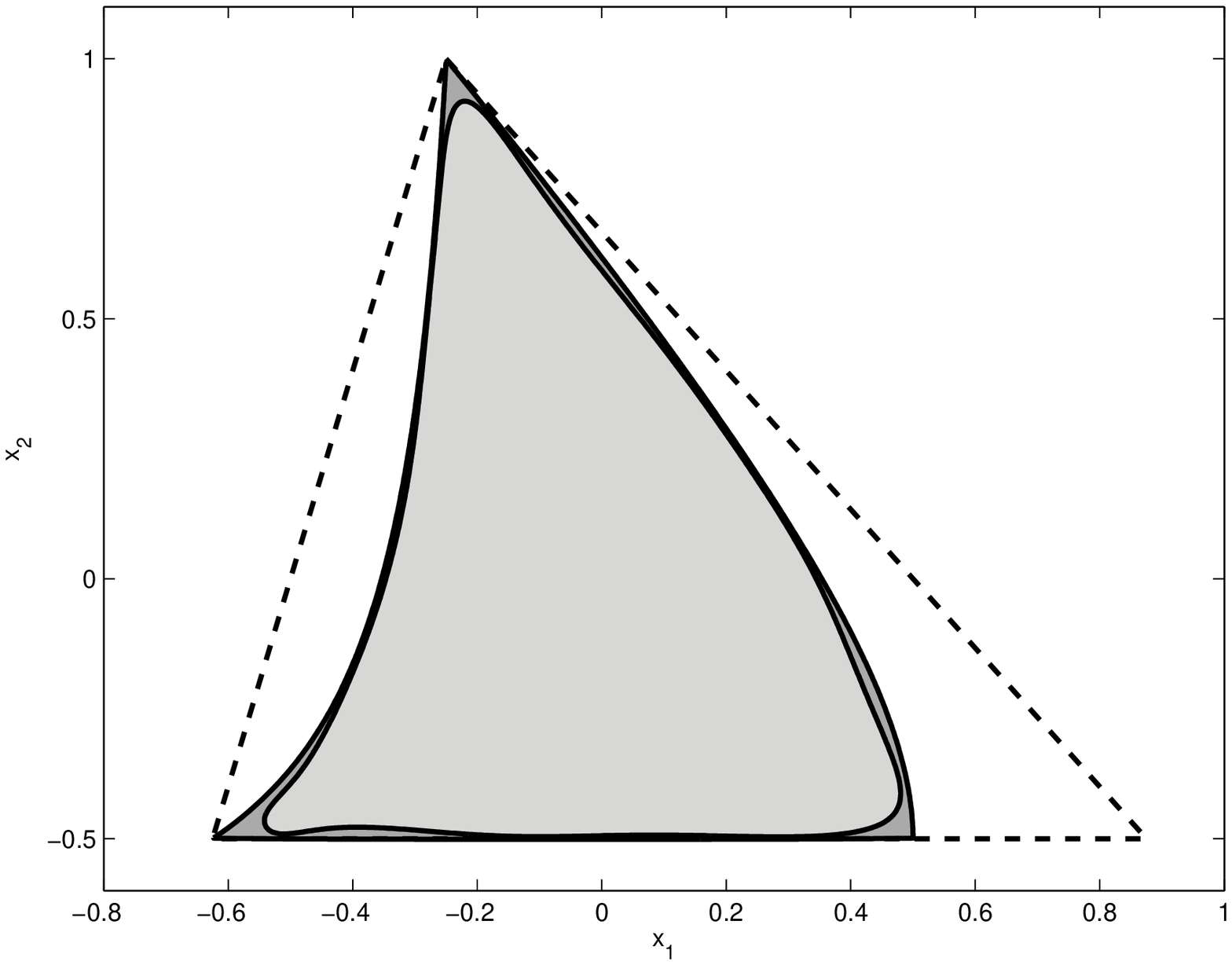}
\caption{Degree six (left) and eight (right) inner approximations (light gray)
of PMI stability region (dark gray) embedded in simplex (dashed).\label{fighermite4b}}
\end{figure}

On Figures \ref{fighermite4a} and \ref{fighermite4b} we represent the degree two, four,
six and eight inner approximations to $\P$, corresponding to stability regions
for the linear system. We observe that the approximations become
tight rather quickly. This is due to the fact that $\B$ is a good outer approximation
of $\P$ with known moments. Tighter outer approximations $\B$ would result in
tighter inner approximations of $\P$, but then the moments of $\B$ can be
hard to compute, see \cite{hls09}.

\begin{figure}[h!]
\centering
\includegraphics[width=0.49\textwidth]{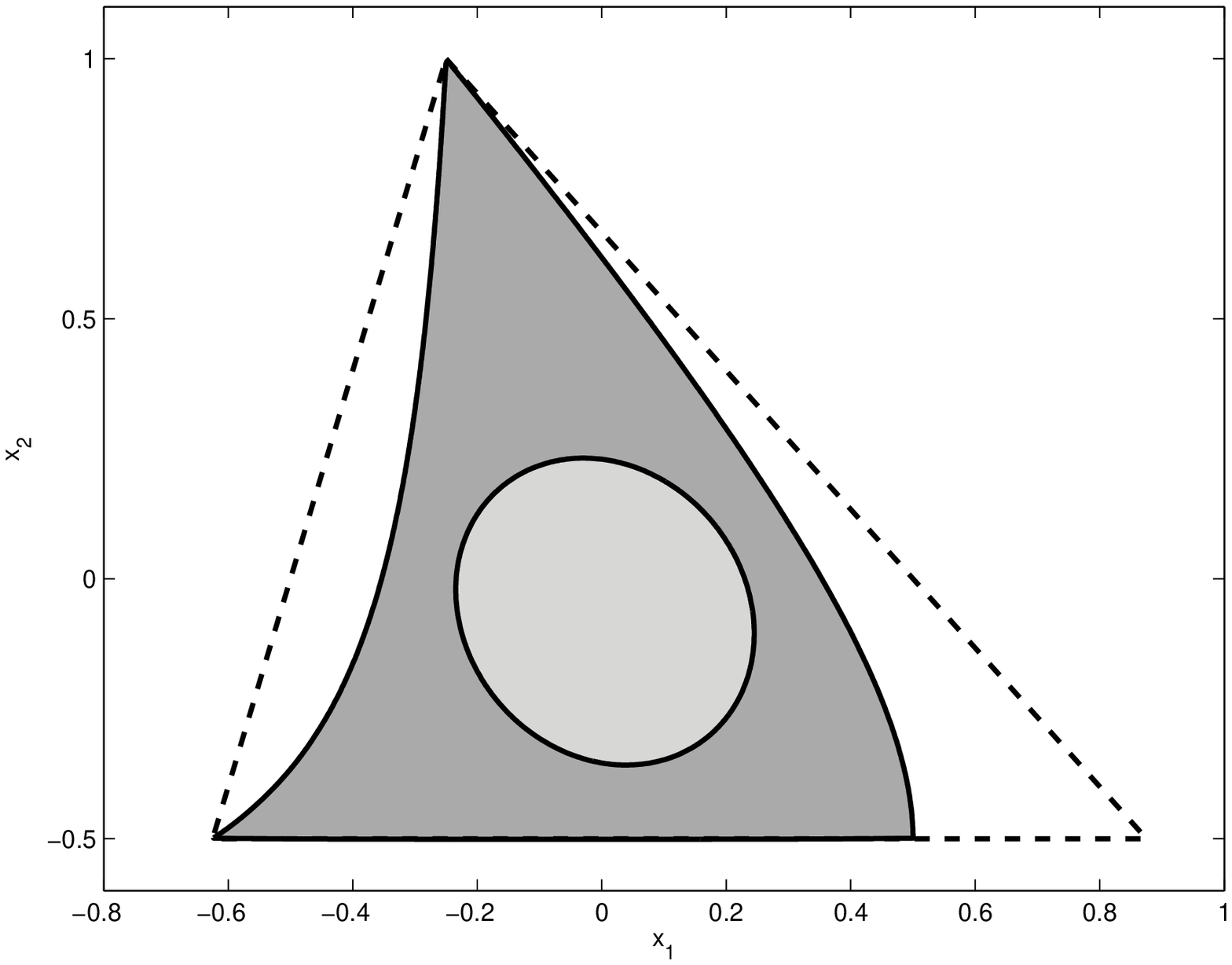}
\includegraphics[width=0.49\textwidth]{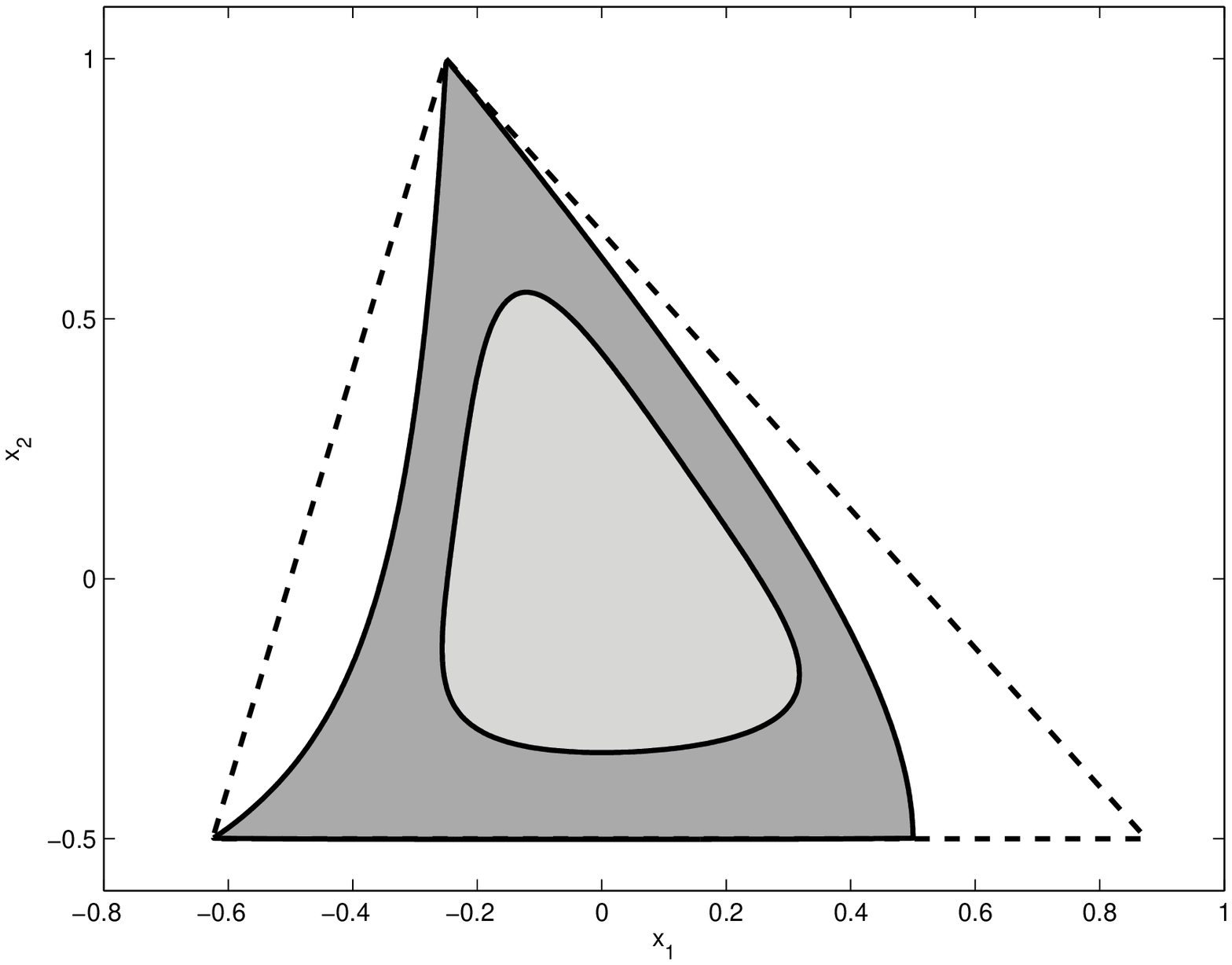}
\caption{Degree two (left) and four (right) inner approximations (light gray)
of robust PMI stability region (dark gray) embedded in simplex (dashed).\label{fighermite4robusta}}
\end{figure}

\subsection{Robust controller design}

Now consider the uncertain polynomial
$z \mapsto x_2+u+2x_1z-(2x_1+x_2)z^3+z^4$ with $u \in {\mathbf U} = \{u \in {\mathbb R} \: :\:
u^2 \leq \frac{1}{16}\}$ with uncertain Hermite matrix $P(x,u)$ 
and the corresponding parametrized PMI stability region $\P$
in (\ref{setp}). Let us use the same bounding set $\B$ 
as in \S \ref{design}.

On Figure \ref{fighermite4robusta}
we represent the degree two and degree four
inner approximations to $\mathbf P$, corresponding to robust stability regions
for the linear system. Comparing with Figure \ref{fighermite4a}
we see that the approximations are smaller, and in particular they do not touch
the stability boundary to cope with the robustness requirements.

\section{Conclusion}

We have constructed a hierarchy of inner approximations of feasible sets
defined by parametrized or uncertain polynomial matrix inequalities (PMI).
Each inner approximation is computed by solving a convex linear
matrix inequality (LMI) problem. The hierarchy converges in a well-defined
analytic sense, so that conservatism of the approximation is guaranteed
to vanish asymptotically. In addition, the inner approximations are simple polynomial
or piecewise-polynomial superlevel sets, so that optimization over these
sets is significantly simpler than optimization over the original
parametrized PMI set. In particular, we remove the possibly complicated
dependence of the problem data on the uncertain parameters.

One may also impose the hierarchy of inner approximations to be nested.
Finally, one may also impose the inner approximations to be convex. In this latter case
they do not converge any more to the feasible set but, on the other hand,  
optimization over the parametrized PMI set can be reformulated
as a convex polynomial optimization problem (of course at the price
of some conservatism). Ideally, beyond convexity, we may also
want the inner convex approximation to be semidefinite representable
(as an explicit affine projection of an affine section of the SDP cone),
and deriving such a representation may be an interesting research direction.

The tradeoff to be found is between tightness of the inner approximation
and degree of the defining polynomials. In the context of robust control
design, a satisfactory inner
approximation can be possibly computed off-line, and then used afterwards
on-line in a feedback control setup. 

Our methodology is valid for general parametrized PMI problems.
However, in the case of parametrized PMI problems coming from
fixed-order robust controller design problems, geometric insight
can be exploited to improve convergence of the hierarchy.
The key information is the knowledge of the moments of the Lebesgue
measure on a compact set which tightly contains the parametrized
PMI set we want to approximate from the inside. In turns out
that for robust control problems this knowledge is available easily,
as illustrated in the paper by several examples.

The main limitation of the approach lies in the ability of solving
primal moment and dual polynomial sum-of-squares LMI problems.
State-of-the-art general-purpose semidefinite programming solvers
can currently address problems of relatively moderate dimensions,
but problem structure and data sparsity can be exploited for
larger problems.

\section*{Acknowledgements}

The first author acknowledges support by 
project No.~103/10/0628 of the Grant Agency of the Czech Republic.
We are grateful to Luca Zaccarian for pointing out a mistake
in a previous version of this paper.

\section{Appendix}\label{technical}

\subsection{Proof of Lemma \ref{lemma1}}

\label{proof-lemma1}
\begin{proof}
The dual of polynomial SOS SDP problem (\ref{sdp}) is moment SDP problem (\ref{sdp*}).
Slater's condition cannot hold for (\ref{sdp*}) because $\V$ has empty interior in $\R^m$. However
it turns out that Slater's condition holds for an equivalent version of SDP problem 
(\ref{sdp*}), i.e., the latter has a strictly feasible solution $\hat{y}$. Indeed, let $J\subset\R[v]$ be the ideal generated by
the polynomial $v\mapsto \theta(v):=1-v^Tv$ so that the real variety $V_\R(J):=\{v\in\R^m: \theta(v)=0\}$ 
associated with $J$ is just the unit sphere $\V$.
It turns out that the real radical\footnote{$\V$ is Zariski dense in $V_{\mathbb{C}}(J)\,(=\{v\in \mathbb{C}^n:\theta(v)=0\})$ so that
$I(V_\R(J))=I(V_\mathbb{C}(J))$. But $\theta$ being irreducible, $J$ is a prime ideal and so $I(V_\mathbb{C}(J))=J$.}
of $J$ is $J$ itself, that is, $I(V_\R(J))=J$ (where for $S\subset\R^m$, $I(S)$ denotes the vanishing ideal).
And after embedding $J$ in $\R[x,u,v]$, we still have $I(V_\R(J))=J$.

Let $H:=\{(\alpha,\beta,\gamma)\in\N^n\times \N^p\times \N^m:\gamma_m\leq 1\}$, and
let $H_d:=\{(\alpha,\beta,\gamma)\in H: \sum_i\alpha_i+\sum_j\beta_j+\sum_\ell \gamma_\ell\leq d\}$.
The monomials $(x^\alpha u^\beta v^\gamma)$, $(\alpha,\beta,\gamma)\in H$, form a basis of the 
quotient space $\R[x,u,v]/J$. Moreover, for every $(\alpha,\beta,\gamma)\in \N^{n+p+m}_d$,
\[x^{\alpha} u^{\beta} v^{\gamma}\,=\,\sum_{(a,b,c)\in H_{d}}p_{abc}\,x^au^{b} v^{c}+ \underbrace{h(x,u,v)}_{\in\R[x,u,v]_{d-2}}\,(1-v^Tv),\]
for some real coefficients $(p_{abc})$, and some $h\in\R[x,u,v]_{d-2}$.
This is because every time one sees a monomial $x^a u^b v^c$ with $c_m\geq2$, one uses
$v_m^2=1-\sum_{j\neq m}v_j^2$ to reduce this monomial modulo $\theta=(1-v^Tv)$. For instance
\begin{eqnarray*}
x^au^bv_1^{c_1}\cdots v_{m-1}^{c_{m-1}}v_m^3&=&x^au^bv_1^{c_1}\cdots v_{m-1}^{c_{m-1}}v_m\times
\underbrace{v_m^2}_{=-\theta+(1-\sum_{j\neq m}v_j^2)}\\
&=&x^au^bv_1^{c_1}\cdots v_{m-1}^{c_{m-1}}v_m
-\sum_{j\neq m}x^au^bv_1^{c_1}\cdots v_j^{c_j+2}\cdots v_{m-1}^{c_{m-1}}v_m\\
&&-\underbrace{x^au^bv_1^{c_1}\cdots v_{m-1}^{c_{m-1}}v_m}_{\in\R[x,u,v]_{d-2}}\,\theta(v),
\end{eqnarray*}
etc. Therefore, for every $(\alpha,\beta,\gamma),(\alpha',\beta',\gamma')
\in H_d$,
\begin{equation}
\label{substitute}
x^{\alpha+\alpha'} u^{\beta+\beta'} v^{\gamma+\gamma'}\,=\,\sum_{(a,b,c)\in H_{2d}}
p_{abc}\, x^au^{b} v^{c}+ \underbrace{h(x,u,v)}_{\in\R[x,u,v]_{2d-2}}\,(1-v^Tv),\end{equation}
for some real coefficients $(p_{abc})$, and some $h\in\R[x,u,v]_{2d-2}$.

So because of the constraints $M_{d-1}((1-v^Tv)\,y)=0$, the semidefinite program (\ref{sdp*}) is equivalent to the semidefinite program:
\begin{equation}
\label{sdp2*}
\begin{array}{rl}
\rho^*_d=\displaystyle\int_\B\lm(x)dx\,-\,\min_y &L_y(v^TP(x,u)v)\\
\mbox{s.t.}&\hat{M}_d(y)\succeq0,\:M_{d-1}((1-v^Tv)\,y)=0\\
&\hat{M}_{d-d_{a_i}}(a_i\,y)\succeq0,\quad i=0,1,\ldots,n_a\\
&\hat{M}_{d-d_{b_j}}(b_j\,y)\succeq0,\quad j=1,\ldots,n_b\\
&L_y(x^\alpha)\,=\,\int_\B x^\alpha\,dx,\quad\forall\alpha\in\N^n_{2d}
\end{array}\end{equation}
where the smaller moment matrix $\hat{M}_d(y)$ is the submatrix of $M_d(y)$ obtained by
looking only at rows and columns indexed in the monomial basis $(x^\alpha y^\beta v^\gamma)$, 
$(\alpha,\beta,\gamma)\in H_d$, instead of $\N^{n+p+m}_d$.
Similarly, the smaller localizing matrix $\hat{M}_{d-d_{a_i}}(a_i\,y)$ is the submatrix of $M_{d-d_{a_i}}(a_i\,y)$ obtained by
looking only at rows and columns indexed in the monomial basis $(x^\alpha y^\beta v^\gamma)$, 
$(\alpha,\beta,\gamma)\in H_{d-d_{a_i}}$, instead of $\N^{n+p+m}_{d-d_{a_i}}$; and similarly for $\hat{M}_{d-d_{b_j}}(b_j\,y)$.

Indeed, in view of (\ref{substitute}) and using $M_d((1-v^Tv)\,y)=0$,
every column of $M_d(y)$ associated with $(\alpha,\beta,\gamma)\in\N^{n+p+m}$ is a linear combination
of columns associated with $(\alpha',\beta',\gamma')\in H_d$. And similary
for $M_{d-d_{a_i}}(a_i\,y)$ and $M_{d-d_{b_j}}(b_j\,y)$. Hence,
$M_d(y)\succeq0\Leftrightarrow \hat{M}_d(y)\succeq0$, and 
\[M_{d-d_{a_i}}(a_i\,y)\succeq0\Leftrightarrow \hat{M}_{d-d_{a_i}}(a_i\,y)\succeq0;\quad
M_{d-d_{b_j}}(b_j\,y)\succeq0\Leftrightarrow \hat{M}_{d-d_{b_j}}(b_j\,y)\succeq0,\]
for all $i=1,\ldots,n_a$, $j=1,\ldots,n_b$.

Next, let $\hat{y}$ be the sequence of moments of the (product) measure $\mu$ uniformly distributed on $\B\times\U\times\V$, and scaled 
so that for all $(\alpha,\beta,\gamma)\in\N^{n+p+m}_{2d}$
\[\hat{y}_{\alpha\beta\gamma}\,=\,\int_{\B\times\U\times\V}x^\alpha\,u^\beta v^\gamma\,d\mu(x,u,v)\,=\,\frac{1}{{\rm vol}\,\U\times \V}
\int_\B\int_\U\int_\V x^\alpha\,u^\beta v^\gamma\,\underbrace{dx\,du\,d\lambda(v)}_{d\mu(x,u,v)}\]
(with $\lambda$ the rotation invariant measure on $\V$).
Therefore, for every $\alpha\in\N^n_{2d}$,
\[\hat{y}_{\alpha 00}\,=\,L_y(x^\alpha)\,=\,\int_{\B\times\U\times\V}x^\alpha\,d\mu(x,u,v)\,=\,\int_\B x^\alpha dx.\]
Moreover, $M_{d-1}((1-v^Tv)\,y)=0$ for every $d$ and importantly,
$\hat{M}_d(\hat{y}) \succ 0$, $\hat{M}_{d-d_{a_i}}(a_i\,\hat{y})\succ0$ and
$\hat{M}_{d-d_{b_j}}(b_j\,\hat{y})\succ0$. To see why, suppose for instance that $h^T\hat{M}_d(\hat{y})h=0$ for some
vector $h\neq0$. This means that for some non trivial polynomial $h\in\R[x,u,v]/J$ of degree $d$,
\[h^T\hat{M}_d(\hat{y})h\,=\,\int_{\B\times\U\times\V}h^2\,d\mu\,=\,0,\]
that is, $h(x,u,v)=0$ for $\mu$-almost all $(x,u,v)\in\B\times\U\times\V$, and so
 $h(x,u,v)=0$ for all $(x,u,v)\in\B\times\U\times\V$ because $h$ is continuous.
But as $\B\times\U$ has nonempty interior in $\R^n\times\R^p$, then necessarily
$h\in I(V_\R(J))\,(=J)$ -- see Lemma \ref{lem-ideal} in section \ref{ideal} -- which contradicts $0\neq h\in\R[x,u,v]/J$.
Therefore $\hat{y}$ is a strictly feasible solution of (\ref{sdp2*}) and so Slater's condition holds for (\ref{sdp2*}).

Denote by $\hat{\Sigma}[x,u,v]_d$ the space of polynomials of degree at most $2d$,
that are SOS of polynomials in $\R[x,u,v]/J$. As $\hat{y}$ is a strictly feasible solution of the semidefinite program (\ref{sdp2*}), by a standard result of convex optimization,
there is no duality gap between  (\ref{sdp2*}) and its dual
\begin{equation}\label{sdp2}
\begin{array}{rl}
\rho'_d\,=\,\displaystyle\int_\B\lm(x)\,dx\,-\,\min_{g,r,s,t} & \displaystyle\int_\B g(x)\,dx\\[1em]
\mathrm{s.t.} & v^T P(x,u) v - g(x) \,=\, r(x,u,v) (1-v^T v) \\
& +\displaystyle\sum_{i=0}^{n_a} s_i(x,u,v) a_i(u)+\displaystyle\sum_{j=1}^{n_b} t_j(x,u,v) b_j(x) \quad\forall (x,u,v)\\
\end{array}
\end{equation}
where now the decision variables are coefficients of
polynomials $g\in\R[x]_{2d}$, $r\in\R[x,u,v]_{2d_r}$, and coefficients of
SOS polynomials $s_i\in\hat{\Sigma}[x,u,v]_{d_{a_i}}$, $i=0,1,\ldots,n_a$,
and $t_j\in\hat{\Sigma}[x,u,v]_{d_{b_j}}$, $j=1,\ldots,n_b$. 
That is, $\rho'_d=\rho_d^*$ and so $\rho_d=\rho^*_d$ because 
$\rho'_d\leq\rho_d\leq\rho^*_d$.
If $\rho'_d<\infty$ then (\ref{sdp2}) is guaranteed to have an optimal solution $(g^*,r^*,s^*,t^*)$.
But observe that such an optimal solution $(g^*,r^*,s^*,t^*)$ is also feasible in (\ref{sdp}), and so having value $\rho'_d=\rho_d=\rho^*_d$,
$(g^*,r^*,s^*,t^*)$ is also an optimal solution of (\ref{sdp}).

It remains to prove that $\rho_d$ is  bounded. For any feasible solution $y$ of (\ref{sdp}), $y_0\leq {\rm vol}\,\B$, and
\begin{equation}
\label{bounds}
L_y(x_i^{2d})\,\leq\,R^{2d}y_0^d\,;\quad L_y(u_j^{2d})\,\leq\,R^{2d}y_0^d\,;\quad L_y(v_k^{2d})\,\leq\,y_0^d,
\end{equation}
for all $i=1,\ldots,n$, $j=1,\ldots,p$, $k=1,\ldots,m$.
This follows from $M_{d-d_{a_{i^*}}}(a_{i^*}\,y)\succeq0$, $M_{d-d_{b_{j^*}}}(b_{j^*}\,y)\succeq0$ and
$M_{d-1}((1-v^Tv)\,y)=0$, where 
$a_{i^*}(x)\,=\,R^2-x^Tx$ and $b_{j^*}(x)\,=\,R^2-u^Tu$; 
see the comments after (\ref{setu}) and (\ref{momb}). Then by \cite[Lemma 4.3]{lasnetzer},
one obtains $\vert y_\alpha\vert\leq R^{2d}({\rm vol}\,\B)^d$, for all $\alpha\in\N^n_{2d}$, which shows that the feasible set of
(\ref{sdp*}) is compact. Hence (\ref{sdp*}) has an optimal solution and $\rho_d$ is finite; therefore
its dual (\ref{sdp}) also has an optimal solution, the desired result.
\end{proof}

\subsection{Proof of Theorem \ref{thmain}}\label{proof-thmain}

\begin{proof}
(a) Let $\K:=\B\times\U\times\V\subset\R^{n+p+m}$ and consider the infinite-dimensional optimization problem
\begin{equation}\label{momp}
\begin{array}{rcll}
\rho & = & \displaystyle\min_{\mu\in M(\K)} & \displaystyle\int_\K v^TP(x,u)v\, d\mu(x,u,v) \\
& & \mathrm{s.t.} & \displaystyle\int_\K x^\alpha d\mu\,=\,\int_\B x^\alpha\,dx,\quad\alpha\in\N^n
\end{array}
\end{equation}
where $M(\K)$ is the space of finite Borel measures on $\K$.
Problem (\ref{momp}) has an optimal solution $\mu^*\in M(\K)$. Indeed,  $\rho\geq\int_\B\lm(x) dx$ because for every $(x,u,v)\in\K$, $v^TP(x,u)v\geq \lm(x)$; and so
for every feasible solution $\mu\in M(\K)$,
\[\int_\K v^TP(x,u,v)v\,d\mu(x,u,v)\,\geq\,\int_\K\lm(x)\,d\mu(x,u,v)\,=\,\int_\B\lm(x)\,dx\]
because $\int_\K x^\alpha d\mu=\int_\B x^\alpha dx$ for all $\alpha\in\N$ and hence
the marginal of $\mu$ on $\R^n$ is the Lebesgue measure on $\B$. On the other hand,
observe that for every $x\in\B$, $\lm(x)=v_x^TP(x,u_x)v_x$ for some $(u_x,v_x)\in \U\times\V$.
Therefore, let $\mu^*\in M(\K)$ be the Borel measure concentrated on $(x,u_x,v_x)$ for all $x\in\B$, i.e.
\[\mu^*(\B'\times\U'\times\V')\,:=\,\int_{\B'\cap\U'}1_{\U'\times\V'}(u_x,v_x)\,dx,\qquad\forall
(\B',\U',\V')\in B(\B)\times B(\U)\times B(\V)\]
where $x\mapsto 1_\B(x)$ denotes the indicator function of set $\B$ and $B(\B)$ denotes the Borel $\sigma$-algebra
of subsets of $\B$.
Then $\mu^*$ is feasible for problem (\ref{momp}) with value
\[
\int_\K v^TP(x,u)v\,d\mu^*(x,u,v)=\int_\B\lm(x)\,dx
\]
which proves that $\rho=\int_\B\lm(x)\,dx$.

Next, $\lm$ being continuous on compact set $\B$, by the Stone-Weierstrass theorem \cite[\S A7.5]{ash}, for every $\varepsilon>0$ there exists
a polynomial $h_\varepsilon\in\R[x]$ such that
\[
\sup_{x\in \B}\vert \lm(x)-h_\varepsilon(x)\vert<\frac{\varepsilon}{2}.
\]
Hence the polynomial
$p_\varepsilon:=h_\varepsilon -\varepsilon$  satisfies $\lm-p_\varepsilon>0$ on $\B$ and so
$v^TP(x,u)v-p_\varepsilon >0$ on $\B\times\U\times\V$. By Putinar's Positivstellensatz, see e.g \cite[Section 2.5]{l09}, there exists
SOS polynomials $r_\varepsilon\in\R[x,u,v]$, and $s_{i\varepsilon},t_{j\varepsilon} \in\Sigma[x,u,v]$ such that equation (\ref{sdp}) is satisfied.
Hence for $d$ sufficiently large, say $d\geq d_\varepsilon$, $(p_\varepsilon,r_\varepsilon,s_{i\varepsilon},t_{j\varepsilon})$
is a feasible solution of (\ref{sdp}) with associated value 
\[\int_\B(\lm(x)-p_\varepsilon(x))\,dx\,\leq\,\frac{3\varepsilon}{2}\int_\B dx.\]
Hence $0\leq\rho_d\leq \frac{3\varepsilon}{2}\int_B dx$ whenever $d\geq d_\varepsilon$
where $\rho_d$ is defined in (\ref{lem1-1}). As $\varepsilon>0$ was arbitrary, we obtain the desired result
\[\lim_{d\to\infty}\rho_d=0.\]
Observe that since $g_d\leq\lm$ for all $d$,
\[
\rho_d=\int_\B(\lm(x)-g_d(x))\,dx=\int_\B\vert \lm(x)-g_d(x)\vert\,dx
\]
so that the convergence $\rho_d\to0$ is just the convergence $g_d\to\lm$ for the $L_1$ norm on $\B$. Finally
the convergence $g_d\to\lm$ in Lebesgue measure on $\B$ follows from \cite[Theorem 2.5.1]{ash}.

(b) For each $x\in\B$, fixed and arbitrary, the sequence $(\bar{g}_d)$ is monotone nondecreasing and bounded above by $\lm$.
Therefore there exists $g^*:\B\to\R$ such that for every $x\in\B$, $\bar{g}_d(x)\uparrow g^*(x)\leq \lm(x)$ as $d\to\infty$.
Since $\bar{g}_d\geq\bar{g}_0$ and $\int_\B \bar{g}_0dx>-\infty$,
by Lebesgue's Dominated Convergence Theorem \cite[\S 1.6.9]{ash}
\[\int_\B g^*(x)dx\,=\,\lim_{d\to\infty}\int_\B \bar{g}_d(x)dx\,=\,\int_\B\lm(x)dx,\]
and so from $g^*(x)\leq\lm(x)$ we deduce that $g^*(x)=\lm(x)$ for almost all $x\in\B$. 
Combining the latter with $\bar{g}_d\uparrow g^*$, we obtain that 
$\bar{g}_d\to \lm$ almost everywhere in $\B$. But then since the Lebesgue measure is finite on $\B$,
by Egorov's theorem \cite[Theorem 2.5.5]{ash}, $\bar{g}_d\to \lm$ almost uniformly in $\B$. Finally, convergence in
Lebesgue measure on $\B$ also follows from \cite[Theorem 2.5.2]{ash}.
\end{proof}

\subsection{Proof of Corollary \ref{coro1}}
\label{proof-coro1}

\begin{proof}
By Theorem \ref{thmain}, $\lim_{d\to\infty}\Vert \lm-g_d\Vert_1=0$. Therefore, by \cite[Theorem 2.5.1]{ash}
the sequence $(g_d)$ converges to $\lm$ in Lebesgue measure, i.e. for every $\varepsilon>0$,
\begin{equation}\label{aux}
\lim_{d\to\infty}
\vol\{x\,:\, \vert\lm(x)-g_d(x)\vert\,\geq\,\varepsilon\}=0.
\end{equation}
Let $\varepsilon>0$ be fixed, arbitrary, and let $\P_\varepsilon:=\{x\in\B\,:\,\lm(x)\geq\varepsilon\}$, so that
$\lim_{\varepsilon\to0}\vol\,\P_\varepsilon = \vol\,\P$. By (\ref{aux}),
$\lim_{d\to\infty}\vol(\P_\varepsilon\cap\{x\in\B\,:\,g_d(x)<0\})=0$. Next, for all $d\in\N$,
\[
\vol\,\P_\varepsilon\,=\,\vol(\P_\varepsilon\cap\{x\in\B\,:\,g_d(x)<0\})
+\vol(\P_\varepsilon\cap\{x\in\B\,:\,g_d(x)\geq0\}).\]
Therefore, taking the limit as $d\to\infty$ yields
\begin{eqnarray*}
\vol\,\P_\varepsilon&=&\underbrace{\lim_{d\to\infty}\vol(\P_\varepsilon\cap\{x\in\B\,:\,g_d(x)<0\})}_{=0\mbox{ by (\ref{aux})}}
+\lim_{d\to\infty}\vol(\P_\varepsilon\cap\underbrace{\{x\in\B\,:\,g_d(x)\geq0\}}_{=\G_d})\\
&=&\lim_{d\to\infty}\vol(\P_\varepsilon\cap\G_d)\,\leq\,\lim_{d\to\infty}\vol\,\G_d.
\end{eqnarray*}
As $\varepsilon>0$ was arbitrary and $\G_d\subset\P$, we obtain the desired result (\ref{coro1-1}).
The proof of (\ref{coro1-2}) is similar.
\end{proof}

\subsection{Proof of Corollary \ref{coro2}}
\label{proof-coro2}

\begin{proof}
Let $0<\varepsilon<\frac{1}{3}$ be fixed, arbitrary. As in the proof of Theorem  \ref{thmain},
for every $k\in\N$ there exists
a polynomial $h_k\in\R[x]$ such that $\sup_{x\in \B}\vert \lm(x)-h_k(x)\vert<\varepsilon^k$. Hence for all $x\in\B$ and all $k\geq1$,
\[\lm(x)-3\varepsilon^k< h_k(x)-2\varepsilon^k<\lm(x)-\varepsilon^k < \lm(x)-3\varepsilon^{k+1}<h_{k+1}(x)-2\varepsilon^{k+1}
<\lm(x)-\varepsilon^{k+1}\]
and so the polynomial $x\mapsto p_k(x):=h_k(x) -2\varepsilon^k$  satisfies $p_{k+1}(x) >p_{k}(x)$ and
$\lm(x)>p_k(x)$ for all $x\in\B$.  Again, by Putinar's Positivstellensatz, see e.g \cite[Section 2.5]{l09}, 
$p_k$ is feasible for (\ref{sdp}) with the additional constraint (\ref{add1}), provided that $d$ is sufficiently large, and with associated value
\[\int_\B\vert \lm(x)-p_k(x)\vert dx\,=\,\int_\B (\lm(x)-p_k(x))dx\,<3\varepsilon^k\int_\B dx\quad\to0\quad\mbox{as }k\to\infty.\]
\end{proof}

\subsection{An auxiliary result for the proof of Lemma \ref{lemma1}}
\label{ideal}

Remember that $J\subset \R[v]$ is the ideal generated by $1-v^Tv$ and the real radical $I(V_\R(J))$ of $J$ is $J$ itself.
And when $J$ is embedded in $\R[x,u,v]$ (with same name of simplicity) we also have $I(V_\R(J))=J$.
\begin{lemma}\label{lem-ideal}
If $f\in\R[x,u,v]_d$ is such that $f(x,u,v)=0$ for all $(x,u,v)\in\B\times\U\times\V$ then $f\in J$.
\end{lemma}
\begin{proof}
Write
\[f(x,u,v)\,=\,\sum_{\alpha\in\N^m_d}g_\alpha(x,u)\,v^\alpha,\]
for some polynomials $(g_\alpha)\subset\R[x,u]_d$, $\alpha\in\N^m_d$. Next, let $(x_0,u_0)\in\B\times\U$ be fixed, so that
$v\mapsto f(x_0,u_0,v)=0$ for all $v\in\V$. Therefore, as a polynomial of $\R[v]$, it vanishes on 
$\V=V_\R(J)$ and as $I(V_\R(J))=J$,  $v\mapsto f(x_0,u_0,v)\in J$, that is,
\begin{equation}
\label{reduction}
f(x_0,u_0,v)\,=\,\sum_{\alpha\in\N^m_d}g_\alpha(x_0,u_0)\,v^\alpha\,=\,(1-v^Tv)\,\theta^{x_0,v_0}(v),\end{equation}
for some polynomial $v\mapsto\theta^{x_0,v_0}(v)\in\R[v]_{d}$. 
The coefficients ($\theta_\alpha^{x_0,u_0})$ of the polynomial $\theta^{x_0,u_0}(v)=\sum_\alpha\theta_\alpha^{x_0,u_0}v^\alpha$ are linear in the coefficients $(g_\beta(x_0,u_0))$, $\beta\in\N^m_d$, of $f$.  Indeed one may reduce each monomial $v^\alpha$ using
$v_m^{2}=1-\sum_{i\neq m}v_i^2$, until there is no monomial $v_m^\beta$ with $\beta>1$. For instance,
\[v_1^{\alpha_1}\cdots v_{m-1}^{\alpha_{m-1}}v_m^{2}=v_1^{\alpha_1}\cdots v_{m-1}^{\alpha_{m-1}}(v^Tv-1)+v_1^{\alpha_1}\cdots v_{m-1}^{\alpha_{m-1}}-\sum_{j\neq m}v_1^{\alpha_1}\cdots v_j^{\alpha_j+2}\cdots v_{m-1}^{\alpha_{m-1}},\]
and
\[v_1^{\alpha_1}\cdots v_{m-1}^{\alpha_{m-1}}v_m^{3}=v_1^{\alpha_1}\cdots v_{m-1}^{\alpha_{m-1}}v_m(v^Tv-1)+v_1^{\alpha_1}\cdots v_{m-1}^{\alpha_{m-1}}v_m-\sum_{j\neq m}v_1^{\alpha_1}\cdots v_j^{\alpha_j+2}\cdots v_{m-1}^{\alpha_{m-1}}v_m,\]
etc., to finally obtain 
\[v^\alpha=p_\alpha(v)(1-v^Tv)+r_\alpha,\qquad \forall \alpha\in \N^m_d,\]
for some $p_\alpha\in\R[v]_{\vert\alpha\vert-2}$ and $r_\alpha\in \R[v]/J$. 
Therefore, summing up over all $\alpha\in\N^m_d$ yields:
\begin{eqnarray}
\nonumber
f(x_0,u_0,v)&=&\sum_{\alpha\in\N^m_d}g_\alpha(x_0,u_0)\,v^\alpha\\
\nonumber
&=&(v^Tv-1)\underbrace{\sum_{\alpha\in\N^m_d}g_\alpha(x_0,u_0)\,p_\alpha (v)}_{h(x_0,u_0,v)}
+\underbrace{\sum_{\alpha\in\N^m_d}g_\alpha(x_0,u_0)\,r_\alpha(v)}_{\mbox{$=0$ as $v\mapsto f(x_0,u_0,v)\in J$}}\\
\label{reduction3}
&=&(v^Tv-1)\,h(x_0,u_0,v),\quad\forall v\in\R^m,
\end{eqnarray}
for some $h\in\R[x,u,v]$. But since (\ref{reduction3}) holds for every $(x,u)\in\B\times\U$, we obtain
\[f(x,u,v)\,=\,(v^Tv-1)\,h(x_0,u_0,v),\quad\forall (x,u,v)\in\B\times\U\times\R^m,\]
and as $\B\times\U\times\R^m$ has nonempty interior, 
\[f(x,u,v)\,=\,(v^Tv-1)\,h(x_0,u_0,v),\quad\forall (x,u,v)\in\R^n\times\R^p\times\R^m,\]
i.e., $f=(v^Tv-1)h$, which proves the desired result that  $f\in J$.
\end{proof}

\end{document}